\documentclass[12pt, a4paper]{amsart}
\usepackage[hmargin=30mm, vmargin=25mm, includefoot, twoside]{geometry}
\usepackage[pdftex]{graphicx}
\usepackage{amssymb,verbatim,amsthm}
\usepackage{latexsym}
\usepackage{mathrsfs}
\usepackage{xspace}
\usepackage{enumerate, paralist}

\usepackage{geometry}
\makeatletter
\def\@tocline#1#2#3#4#5#6#7{\relax
  \ifnum #1>\c@tocdepth 
  \else
    \par \addpenalty\@secpenalty\addvspace{#2}%
    \begingroup \hyphenpenalty\@M
    \@ifempty{#4}{%
      \@tempdima\csname r@tocindent\number#1\endcsname\relax
    }{%
      \@tempdima#4\relax
    }%
    \parindent\z@ \leftskip#3\relax \advance\leftskip\@tempdima\relax
    \rightskip\@pnumwidth plus4em \parfillskip-\@pnumwidth
    #5\leavevmode\hskip-\@tempdima
      \ifcase #1
       \or\or \hskip 1em \or \hskip 2em \else \hskip 3em \fi%
      #6\nobreak\relax
    \dotfill\hbox to\@pnumwidth{\@tocpagenum{#7}}\par
    \nobreak
    \endgroup
  \fi}
\makeatother
\makeatletter
\newtheorem*{rep@thm}{\rep@title}
\newcommand{\newrepthm}[2]{%
\newenvironment{rep#1}[1]{%
 \def\rep@title{#2 \ref{##1}}%
 \begin{rep@thm}}%
 {\end{rep@thm}}}
\makeatother

\newtheorem{prop}{Proposition}[section]
\newtheorem{thm}[prop]{Theorem}
\newtheorem{cor}[prop]{Corollary}
\newtheorem{conj}[prop]{Conjecture}
\newtheorem{lem}[prop]{Lemma}
\theoremstyle{definition}
\newtheorem{defn}[prop]{Definition}
\newtheorem{remark}[prop]{Remark}
\newtheorem{que}[prop]{Question}
\theoremstyle{remark}
\newtheorem{example}[prop]{Example}
\numberwithin{equation}{section}

\newcommand{\moins}{\,-\!\!\!\!-\,}
\newcommand{\moinss}{-\!\!\!-}
\newcommand{\mk}{\mathfrak}
\newcommand{\Out}{\textnormal{Out}}
\newcommand{\Aut}{\textnormal{Aut}}
\newcommand{\Nil}{\textnormal{Nil}}
\newcommand{\Nilc}{\textnormal{Nilc}}
\newcommand{\Par}{\mathsf{Ve}}

\newcommand{\GL}{\mathrm{GL}}
\newcommand{\SL}{\mathrm{SL}}
\newcommand{\FT}{\textnormal{FT}}
\newcommand{\lp}{(\!(}
\newcommand{\rp}{)\!)}
\newcommand{\eps}{\varepsilon}
\DeclareMathOperator{\Ker}{Ker}
\DeclareMathOperator{\Hom}{Hom} 
\DeclareMathOperator{\Isom}{Isom}
\newcommand{\bd}{\partial} 

\begin{document}

\title{Commability and focal locally compact groups}
%
\author{Yves Cornulier}
\address{Laboratoire de Math\'ematiques\\
B\^atiment 425, Universit\'e Paris-Sud 11\\
91405 Orsay\\ France}
\email{yves.cornulier@math.u-psud.fr}

\date{February 2, 2015 with incorporation of errata in \S 2, December 6, 2017}
%










\begin{abstract}
We introduce the notion of commability between locally compact groups, namely the equivalence relation generated by cocompact inclusions and quotients by compact normal subgroups. We give a classification of focal hyperbolic locally compact groups up to commability. In the mixed case, it involves a real parameter, which is shown to be a quasi-isometry invariant.
\end{abstract}

\maketitle

\section{Introduction}

Classically, two groups are called {\em commensurable} if they have isomorphic finite index subgroups, and {\em commensurable up to finite kernel} if they have finite index subgroups that are isomorphic after modding out by a finite normal subgroup. It is easy to check that these are transitive (and hence equivalence) relations. In the setting of locally compact groups, it is natural to consider the same relations by only considering open finite index subgroups and topological isomorphisms. However, another variant is to replace finite index open subgroups by cocompact closed subgroups, and compact kernels. This is a natural setting, in view of the fact that cocompact closed embeddings and quotients by compact normal subgroups are coarse equivalences. A difference with the previous setting, however, is that the intersection of two closed cocompact subgroups may not be cocompact: for instance consider the lattices $\mathbf{Z}$ and $\sqrt{2}\mathbf{Z}$ in $\mathbf{R}$. A consequence is that the proof that these relations are transitive falls down, and they are actually not transitive (this follows from \cite[Corollary 10]{MSW}). It is still natural to consider the equivalence relations they generate. 

We abbreviate locally compact group into {\em LC-group}. For a homomorphism between LC-groups, write {\em copci} as a (pronounceable) shorthand for {\em continuous proper with cocompact image}. A copci homomorphism $f:G\to H$ obviously factors as the composition $G\to G/\Ker(f)\to H$ of the quotient map by a compact normal subgroup and an injective copci, and the latter can be viewed as a topological isomorphism onto a cocompact closed subgroup.

\begin{defn}\label{d_commable}Let us say that two LC-groups $G,H$ are {\em commable} if there exist an integer $k$ and a sequence (called commability) of copci homomorphisms

\begin{equation} G=G_0\moins G_1\moins G_2\moins\dots\moins G_k=H,\label{ggih}\end{equation}
where each sign $\moinss$ denotes an arrow in either direction.

If all the above homomorphisms are required to be injective, we call $G$ and $H$ {\em strictly commable}. If all $G_i$ are required to belong to a certain class $\mathcal{C}$ of groups, we say that $G$ and $H$ are commable (or strictly commable) {\em within} $\mathcal{C}$.
\end{defn}

For instance, for discrete groups, commensurability up to finite kernels (resp.\ commensurability) is the same as commability (resp.\ strict commability) within discrete groups. On the other hand, many discrete groups, including semidirect products $\mathbf{Z}^2\rtimes\mathbf{Z}$ occurring as lattices in SOL, are commable but not commensurable.

Here are three (related) classical motivations for the study of commability:
\begin{itemize}
\item The quasi-isometric classification. Indeed, two commable locally compact groups are always quasi-isometric, and many known examples of quasi-isometric groups are actually commable by construction. A natural question, asked in an early version of \cite{CQI} is to find two compactly generated locally compact groups (e.g., discrete) that are quasi-isometric but not commable. Carette and Tessera checked that if $\Gamma_1,\Gamma_2$ are non-commensurable cocompact lattices in, say $\SL_2(\mathbf{C})$, then $\Gamma_1\ast\mathbf{Z}$ and $\Gamma_2\ast\mathbf{Z}$ are not commable although they are quasi-isometric.
\item The so-called study of ``locally compact envelopes". In the setting introduced by Furman \cite{Fur}, it consists, given some finitely generated group $\Gamma$, of classifying locally compact groups in which $\Gamma$ admits an embedding as a cocompact lattice. However, in the setting of \cite{Fur}, it seems that part of the methods essentially use only the quasi-isometry class and might be modified to give a full classification of the whole quasi-isometry class (within locally compact groups). In the realm of solvable groups, further examples were recently obtained by Dymarz \cite{DyE}.
\item The study of ``model spaces". Indeed if $G=\Isom(X)$ for some proper metric space $X$ with a cocompact isometry group, then a copci homomorphism $H\to G$ is the same as a continuous proper cocompact isometric action of $H$ on $X$. Thus two $\sigma$-compact LC-groups $H_1,H_2$ admit copci homomorphisms into the same LC-group if and only if they have a common ``model space"; this is a strong form of commability.
Certain quasi-isometry classes of groups can be described as those
groups admitting a continuous proper cocompact isometric action on a certain ``model space" $X$. Here we rather emphasize the group $G$, because the space $X$ is sometimes less canonical than the group itself.
\end{itemize}

Section \ref{opr} focuses on generalities about commability, especially related to the relevance of the necessity of modding out by compact normal subgroups. It therefore involves a general discussion on the polycompact radical $\mathsf{W}(G)$ of a locally compact group $G$, namely the union of its compact normal subgroups. In many cases, this subgroup is compact, but it can also fail to be compact, in which case its closure is also non-compact; it can also even fail to be closed; however we use a difficult result of Trofimov to show:

\begin{thm}\label{wclo}If $G$ is a {\em compactly generated} locally compact group, then $\mathsf{W}(G)$ is a closed subgroup.
\end{thm}

Examples of compactly generated locally compact groups with $\mathsf{W}(G)$ not compact are easy to exhibit (e.g., finitely generated groups with an infinite, locally finite central subgroup); we provide, in \S\ref{spc}, a more difficult example of a pair of cocompact lattices $\Gamma_1,\Gamma_2$ in the same compactly generated locally compact group (thus $\Gamma_1$ and $\Gamma_2$ are finitely generated) such that $\mathsf{W}(\Gamma_1)$ is finite (actually trivial) and $\mathsf{W}(\Gamma_2)$ is infinite.

We now especially focus on the study of commability in the context of focal LC-groups, which we abbreviate here as {\em focal LC-groups}. Since in this paper we favor an ``algebraic" point of view rather than geometric, we use the following definition for focal groups.

\begin{defn}\label{defocal}
Let $\alpha$ be a self-homeomorphism of a Hausdorff topological space $X$. We say that $\alpha$ is {\em compacting} if there exists a compact subset $\Omega\subset X$ (called vacuum subset) such that for every compact subset $K$ of $X$, there exists $n\ge 1$ such that $\alpha^n(K)\subset\Omega$. 

We say that an LC-group $G$ is a {\em focal LC-group} if it is topologically isomorphic to a topological semidirect product $N\rtimes\Lambda$ with $\Lambda\in\{\mathbf{Z},\mathbf{R}\}$ and $N$ noncompact, such that the action by conjugation of each positive element of $\Lambda$ induces a compacting automorphism of $N$.
\end{defn}

\begin{example}\label{exvi}If $(V_i)_{i=1,\dots,n}$ are finite-dimensional vector spaces over nondiscrete locally compact normed fields $\mathbf{K}_i$, and $\phi_i$ is an automorphism of $V_i$ with spectral radius $<1$, then $\phi:(v_1,\dots,v_k)\mapsto (\phi_1(v_1),\dots,\phi_k(v_k))$ is a compacting automorphism of $V=\prod_i V_i$. (It is actually {\em contracting}, in the sense that every neighborhood of zero is a vacuum subset.)
\end{example}

It was established in \cite{CCMT} that an LC-group is focal if and only if it is non-elementary Gromov-hyperbolic and amenable.

Given a focal LC-group $G$, we are mainly interested in the description of those locally compact groups $H$ that are commable to $G$. This question splits into the internal and external parts, according to whether $H$ is focal or not.

If $G=N\rtimes\Lambda$ is a focal group, a straightforward argument shows that the kernel of the modular homomorphism $\Delta_G$ is $N$; in particular, $N$ is a characteristic subgroup of $G$, and is also the kernel of any nontrivial homomorphism into $\mathbf{R}$. Denote by $N^\circ$ the identity component of $N$.

\begin{defn}We say that the focal group $G$ is 
\begin{itemize}
\item of {\em connected type} if $N/N^\circ$ is compact;
\item of {\em totally disconnected type} if $N^\circ$ is compact;
\item of {\em mixed type} otherwise, i.e.\ if both $N^\circ$ and $N/N^\circ$ are noncompact.
\end{itemize}
\end{defn}

The type of a focal group is a commability invariant. 
The main purpose of this paper is to classify commable focal groups LC-groups up to commability within focal LC-groups.

For every focal LC-group $\mathsf{W}(G)$ is compact and is thus the maximal compact normal subgroup of $G$; if $\mathsf{W}(G)=1$ we say that $G$ is W-faithful.

We say that a focal LC-group $G$ is focal-universal if $\mathsf{W}(G)=1$ and for every focal group $H$ commable to $G$ within focal groups, there is a copci homomorphism $H\to G$. 

\begin{thm}
If $G$ is a focal LC-group of connected type, then its commability class within focal groups contains, up to isomorphism, a unique focal-universal LC-group $\hat{G}$. Thus, a focal LC-group $H$ is commable to $G$ within focal groups if and only if $H$ admits a copci homomorphism into $\hat{G}$. 
\end{thm}

\begin{example}
If $G$ is a closed cocompact subgroup in $\Isom(X)$, where $X$ is a rank~1 symmetric space of noncompact type and suppose that $G$ fixes a point $\omega$ on the boundary $\partial X$. Then $G$ is focal of connected type and $\hat{G}$ can be identified to the stabilizer $\Isom(X)_\omega$. For instance, when $X$ is a real $d$-dimensional hyperbolic space ($d\ge 2$), we can choose $G$ so that it can be identified to the group of affine homotheties of the Euclidean space $\mathbf{R}^{d-1}$, and then $\hat{G}$ can be identified to the group of affine similarities of $\mathbf{R}^{d-1}$.

Another example is when $G=G_\lambda$ ($\lambda\ge 1$) is the semidirect product $\mathbf{R}^2\rtimes\mathbf{R}$, where $t\in\mathbf{R}$ acts on $\mathbf{R}^2$ by $t\cdot (x,y)=(e^tx,e^{\lambda t}y)$. If $\lambda=1$ this is a particular instance of the previous case; now assume $\lambda>1$. Then $\hat{G}$ can be identified to the semidirect product $\mathbf{R}^2\rtimes (\mathbf{R}\times\{\pm 1\}^2)$, where the action of $\mathbf{R}$ is the same, and the action of $(\eps_1,\eps_2)\in \{\pm 1\}^2$ is given by $(\eps_1,\eps_2)\cdot (x,y)=(\eps_1x,\eps_2y)$. Thus $G_\lambda$ has index 4 in $\widehat{G_\lambda}$ for $\lambda>1$ (while $G_1$ has infinite index, and actually positive codimension, in $\widehat{G_1}$).
\end{example}

This is, in a sense, the best possible situation. Indeed, it can be checked (Corollary \ref{nouni}) that no focal LC-group not of connected type is focal universal. Still, the commability classification can be made, using suitable invariants. 

First, for $G$ focal of totally disconnected type, the range of its modular function is the multiplicative group generated by some integer $s_G\ge 2$. Define $q_G$ as the minimal root of $s_G$, that is, the minimal integer $\ge 2$ having $s_G$ as an integral power.

\begin{example}
Fix a prime $p$. Let $\phi$ be an automorphism of $\mathbf{Q}_p^k$ of spectral radius $<1$, an consider the semidirect product $G=G_\phi=\mathbf{Q}_p^k\rtimes_\phi\mathbf{Z}$, where the positive generator of $\mathbf{Z}$ acts by $\phi$ (this is a particular case of Example \ref{exvi} with $n=1$). Then $s_G=p^\ell$, where $\ell$ is the $p$-valuation of the determinant of $\phi$, and $q_G=p$.
\end{example}

\begin{prop}\label{inqtd}
Two focal LC-groups $G,H$ of totally disconnected type are commable within focal groups if and only if $q_G=q_H$ (equivalently, $s_G$ and $s_H$ have a common integral power). 
\end{prop}

If $G$ is focal of mixed type, then $G/G^\circ$ is focal of totally disconnected type and we define $q_G=q_{G/G^\circ}$. The group $G/G^\circ$ can be thought of as the totally disconnected side of $G$. Actually, $G$ also admits a connected side, which can also be viewed as a quotient $G/G^\sharp$ of $G$, which is focal of connected type. Here $G^\sharp$ is the elliptic radical of $G$, that is, the largest closed normal subgroup in which every compact subset is contained in a compact subgroup. It is no surprise that both the commability classes of $G/G^\circ$ and $G/G^\sharp$ are commability invariants of $G$. Interestingly, there is a third numerical invariant, which relates both sides by comparing the speed of contraction in the connected and totally disconnected parts. 

The modular functions of $G/G^\circ$ and $G/G^\sharp$ define, by composition, homomorphisms $\Delta_G^{\textnormal{td}}$, $\Delta_G^{\textnormal{con}}:G\to\mathbf{R}_+$, and there is a unique number $\varpi_G>0$ such that $\Delta_G^{\textnormal{td}}=(\Delta_G^{\textnormal{con}})^{\varpi_G}$; see Definition \ref{dvarpi} for more details. (By extension, it is natural to define $\varpi_G=0$ if $G$ is of connected type and $\varpi_G=+\infty$ if $G$ is totally disconnected type.)

\begin{example}
Fix a prime $p$; let $s$, $t$ be
elements of norm in $\mathopen]0,1\mathclose[$ of $\mathbf{R}$ and $\mathbf{Q}_p$ respectively. Consider the semidirect product $G=G_{s,t}=(\mathbf{R}\times\mathbf{Q}_p)\rtimes\mathbf{Z}$, where the positive generator of $\mathbf{Z}$ acts on the ring $\mathbf{R}\times\mathbf{Q}_p$ by multiplication by $(s,t)$ (this is a particular case of Example \ref{exvi} with $n=2$). Then $\varpi_G=v_p(t)/\log_p(|s|)$. If we do the same with $\mathbf{R}$ replaced by $\mathbf{C}$, then $\varpi_G=v_p(t)/2\log_p(|s|)$ (the point being that the multiplication by $s$ multiplies the Haar measure of $\mathbf{C}$ by $|s|^2$).  
\end{example}

\begin{thm}
Let $G,H$ be focal LC-groups of mixed type. Then $G$ and $H$ are commable within focal groups if and only if the following three conditions are fulfilled
\begin{enumerate}
\item $G/G^\sharp$ and $H/H^\sharp$ are commable (i.e., $\widehat{G/G^\sharp}$ and $\widehat{H/H^\sharp}$ are isomorphic);
\item $G/G^\circ$ and $H/H^\circ$ are commable (i.e., $q_G=q_G$);
\item $\varpi_G=\varpi_H$.
\end{enumerate}
\end{thm}

The three above results concern commability within focal groups. It turns out that for most focal groups, the commability class actually consists of focal groups and thus commability within focal groups and commability are just the same.

\begin{prop}\label{inconfo}
A focal LC-group $G$ is commable to a non-focal LC-group exactly in the following two cases:
\begin{itemize}
\item $G$ is of totally disconnected type;
\item $G$ admits a copci homomorphism into the group of isometries of a rank 1 symmetric space of noncompact type.
\end{itemize}
\end{prop}

To complete the commability classification of focal LC-groups, it is therefore enough to consider the groups in Proposition \ref{inconfo}. An easy result (Proposition \ref{tdck}) contrasting with Proposition \ref{inqtd} is that all focal LC-groups of totally disconnected type are commable.
To deal with the connected type, an argument based on quasi-isometries is needed, namely Mostow's result that non-homothetic symmetric spaces of rank 1 and noncompact type are pairwise non-quasi-isometric. An even simpler argument based on the topology of the boundary shows that focal groups of different types are not quasi-isometric and hence not commable. Using this and Proposition \ref{inconfo}, we deduce

\begin{prop}
Let $G,H$ be focal LC-groups. Then $G$ and $H$ are commable if and only if either
\begin{itemize}
\item $G$ and $H$ are commable within focal groups, or
\item $G$ and $H$ both have totally disconnected type.
\end{itemize}
\end{prop}

This completes the commability classification, but still it does not gives a full description of commability classes in the special cases.
Indeed, this requires a locally compact version of known quasi-isometric results. This results in: if $X$ is a symmetric space of rank 1 and noncompact type, then a locally compact group is quasi-isometric to $X$ if and only if admits a copci homomorphism into $\Isom(X)$. See the discussion in \cite{CQI}, which also includes a discussion in the totally disconnected type case. 

We finally come to the fact that the invariant $\varpi$ is a quasi-isometry invariant.

\begin{thm}\label{th_mii}
If $G_1,G_2$ are quasi-isometric focal groups, then $\varpi(G_1)=\varpi(G_2)$.
\end{thm}

The proof uses the topology of the boundary of those groups, as well as a double use of the invariance of the vanishing of $L^p$-cohomology in degree 1. On the other hand, T.~Dymarz \cite{Dy12} recently proved that the non-power integer $q_G$ is a quasi-isometry invariant of the focal group of mixed type $G$. Combined with the previous theorem, this provides some partial results towards the quasi-isometry classification of focal groups.

\begin{conj}\label{fconj}
Let $G$ be a focal LC-group. Then every LC-group quasi-isometric to $G$ is commable to $G$.
\end{conj}

Conjecture \ref{fconj} holds for $G$ if $G$ has totally disconnected type, the proof making use of a surprisingly involved construction of
Carette \cite{car}; see Question \ref{qtre}. In Section \ref{qivpi}, we show, using Theorem \ref{th_mii} and Dymarz' QI-invariance of $q$ theorem, that the conjecture in the mixed type reduces to the connected type, as well as some explicit instances where the latter holds. Conjecture \ref{fconj} is more thoroughly discussed in \cite{CQI}.

Further consequences of Theorem \ref{th_mii} are developed in Section~\ref{qivpi}.

\setcounter{tocdepth}{1}    
\tableofcontents

\noindent {\bf Acknowledgements.} I thank Tullia Dymarz and Pierre de la Harpe, as well as the referee, for useful remarks and corrections.

\section{On the polycompact radical}\label{opr}


\noindent {\footnotesize Errata for this section have been incorporated after publication; they are mentioned in footnotes.}

\subsection{Generalities}

The following definition summarizes and extends notions from the introduction.

\begin{defn}\label{dpr}
Let $G$ be a locally compact group. Define its {\em polycompact radical} (or {\em polyfinite radical} if $G$ is discrete) and denote by $\mathsf{W}(G)$ the union of all compact normal subgroups of $G$. If $\mathsf{W}(G)=\{1\}$ we say that $G$ is W-{\em faithful}. 
Let $\mathsf{B}(G)$ (sometimes referred as topological FC-center) be the set of elements in $G$ whose conjugacy class has a compact closure.
\end{defn}

Note that $\mathsf{W}(G)$ is compact if and only if it has a compact closure, and if this holds then $G/\mathsf{W}(G)$ is W-faithful. The letter $\mathsf{B}$ stands for bounded, since it corresponds to elements of $G$ inducing a bounded inner automorphism $\alpha$, namely such that $\{g^{-1}\alpha(g)\mid g\in G\}$ has a compact closure.

Recall that a locally compact group $G$ is {\em elliptic} (or {\em locally elliptic}) if every compact subset is contained in a compact subgroup. An element $g\in G$ is {\em elliptic} if it is contained in some compact subgroup.

\begin{defn}
The {\em elliptic radical} of $G$ is the union $G^\sharp$ of all its normal closed elliptic subgroups. 
\end{defn}

By a theorem of Platonov \cite{Pl66}, $G^\sharp$ is a closed elliptic normal subgroup.

\begin{example}\label{yamabe}
By a theorem of Yamabe \cite[Theorem 4.6]{MZ}, if $G$ is connected-by-compact (i.e.\ $G/G^\circ$ is compact, where $G^\circ$ is the identity component), then $\mathsf{W}(G)=G^\sharp$ is compact and $G/\mathsf{W}(G)$ is a Lie group with finitely many connected components. In this case, the subgroup $\mathsf{B}(G)$ is closed and described by Tits in \cite{Tits}.
\end{example}

The class of elliptic LC-groups is stable under taking extensions. In particular, $G/G^\sharp$ is W-faithful for every locally compact group $G$. Also, for any locally compact group $G$ we have the inclusion $\mathsf{W}(G)\subset \mathsf{B}(G)\cap G^\sharp$. Of course $\mathsf{B}(G)$ need not be contained in $G^\sharp$, but is not much larger, by the following fact, relying on results of Wu-Yu \cite{WY}:

\begin{prop}\label{pwy}
Let $G$ be a locally compact group.
\begin{enumerate}[(1)]
\item\label{wy1} the quotient $\overline{\mathsf{B}(G)}/\overline{\mathsf{B}(G)}^\sharp$ is abelian and isomorphic to the direct product of a vector group (i.e., isomorphic to $\mathbf{R}^d$ for some $d$) and a discrete torsion-free abelian group:
\item\label{wy2} the set of elliptic elements in $\mathsf{B}(G)$ is equal to $\mathsf{W}(G)$. In particular, $\mathsf{B}(G)\cap G^\sharp=\mathsf{W}(G)$ and if $\mathsf{B}(G)$ is closed, then so is $\mathsf{W}(G)$.
\item 

\hspace{2cm}{\rm ---This item and the next one are erased}\footnote{The published version of this proposition was flawed in its items (3) and (4). It included the assertion $\overline{\mathsf{B}(G)}=\overline{\mathsf{B}(G)}^\circ\overline{\mathsf{W}(G)}$ for which $G=\mathbf{Z}$ is a trivial counterexample, and more anecdotally contained the assertion that $\overline{\mathsf{B}(G)}^\circ$ is nilpotent, for which $G=\mathrm{SO}(3)$ is a counterexample. The proof of the first wrong statement relied on a misquotation of \cite{WY}. The current (\ref{wy3}), (\ref{wy4}) (\ref{wyz}) are essentially borrowed from correct statements in (3) and (4) of the published version: namely the only new statement is the last one in (\ref{wy3}), correcting an incorrect one. We have shifted the enumeration to avoid ambiguity. While my initial proof of these facts was roughly correct, the proof of the important Theorem \ref{wclo2} made use of a fake assertion in this proposition; it is corrected here. I have also slightly completed (\ref{wy2}) and added (\ref{wya}) and (\ref{wyb}); (\ref{wya}) is used in the corrected proof of Theorem \ref{wclo2}.
}---
\item


\item\label{wy3} 


The intersection $\overline{\mathsf{B}(G)}^\circ\cap\overline{\mathsf{B}(G)}^\sharp$ is equal to $\mathsf{W}(G^\circ)^\circ$ and in particular is compact and connected. Modulo their compact intersection $\mathsf{W}(G^\circ)^\circ$, $\overline{\mathsf{B}(G)}^\circ$ and $\overline{\mathsf{B}(G)}^\sharp$ generate their topological direct product.
The quotient $\overline{\mathsf{B}(G)}^\circ/\mathsf{W}(G^\circ)^\circ$ is a vector group, and the quotient $\overline{\mathsf{B}(G)}/\overline{\mathsf{B}(G)}^\circ\overline{\mathsf{B}(G)}^\sharp$ is discrete torsion-free abelian.



\item\label{wy4} We have $\overline{\mathsf{W}(G)}=\overline{\mathsf{B}(G)}^\sharp=\overline{\mathsf{B}(G)}\cap G^\sharp$.
 
\item\label{wyz} We have $\overline{\mathsf{B}(G)}^\circ=\mathsf{B}(G)^\circ$.
 
\item\label{wya}We have $\overline{\mathsf{B}(G)}=\mathsf{B}(G)\overline{\mathsf{B}(G)}^\sharp=\mathsf{B}(G)\overline{\mathsf{W}(G)}$, and more generally $\overline{\mathsf{B}(G)}=\mathsf{B}(G)U$ for every open subgroup $U$ of $\overline{\mathsf{B}(G)}^\sharp$. Also the converse of the implication in (\ref{wy2}) holds: if $\mathsf{W}(G)$ is closed, so is $\mathsf{B}(G)$.

\item\label{wyb}We have $\overline{\mathsf{B}(G)}\cap G^0\subset\mathsf{B}(G)$.
 




\end{enumerate}
\end{prop}




\begin{proof}
Define $H=\overline{\mathsf{B}(G)}$. Then $\mathsf{B}(H)$ contains $\mathsf{B}(G)$ and hence is dense in $H$, which is the condition under which the results of \cite{WY} can be applied.

By \cite[Theorem 4]{WY}, the set $E$ of elliptic elements in $H$ is a closed subgroup and $H/E$ is abelian with the given structure, proving (\ref{wy1}), as well as the fact that the set of elliptic elements in $\overline{\mathsf{B}(G)}$ is equal to $\overline{\mathsf{B}(G)}\cap G^\sharp$, which is part of (\ref{wy2}).


Clearly $\mathsf{B}(G)\cap G^\sharp\supset \mathsf{W}(G)$; conversely, if an element $g$ belongs to $\mathsf{B}(G)\cap G^\sharp$, let $\overline{C_g}$ be the closure of its conjugacy class; then $L$ is compact and contained in $G^\sharp$, so the closed subgroup it generates is compact and normal in $G$, and hence $g\in \mathsf{W}(G)$. This proves (\ref{wy2}).

Let us now prove (\ref{wy3}). 
To prove the intersection property and the assertion on $\overline{\mathsf{B}(G)}^\circ/\mathsf{W}(G^\circ)^\circ$, can suppose that $\mathsf{W}(G^\circ)^\circ$ is trivial, and we have to check that $\overline{\mathsf{B}(G)}^\circ$ is a vector group, and that $\overline{\mathsf{B}(G)}^\circ\cap\overline{\mathsf{B}(G)}^\sharp=\{1\}$. 
Then $\overline{\mathsf{B}(G)}^\circ$ is compact-by-(vector group). Since $\mathsf{W}(G^\circ)^\circ$ is trivial, $\overline{\mathsf{B}(G)}^\circ$ is actually profinite-by-(vector group); it is also pro-Lie and connected. Hence if by contradiction the profinite kernel is nontrivial, then $\overline{\mathsf{B}(G)}^\circ$ has a quotient that is (nontrivial finite)-by-(vector group) and connected. This would contradict simple connectedness of the vector group. So $\overline{\mathsf{B}(G)}^\circ$ is a vector group. Hence it indeed has trivial intersection with $\overline{\mathsf{B}(G)}^\sharp$, which consists of elliptic elements.

To complete the topological direct product assertion, since $\overline{\mathsf{B}(G)}^\circ$ is contained in $G^\circ$, it is enough to prove that the quotient map $\pi:G\to G/G^\circ$ is proper in restriction to $\overline{\mathsf{B}(G)}^\sharp$. Showing this amounts to checking that the inverse image of some compact open subgroup $K\subset G/G^\circ$ is compact. Indeed, this inverse image is $\overline{\mathsf{B}(G)}^\sharp\cap\pi^{-1}(K)$, and is contained in $(\pi^{-1}(K))^\sharp$. Since $\pi^{-1}(K)$ is connected-by-compact, its elliptic radical is compact. Hence $\overline{\mathsf{B}(G)}^\sharp\cap\pi^{-1}(K)$ is indeed compact.

Then $\overline{\mathsf{B}(G)}/\overline{\mathsf{B}(G)}^\circ\overline{\mathsf{B}(G)}^\sharp$ is discrete abelian, as a consequence of (\ref{wy1}). It is torsion-free because it embeds into the product $\overline{\mathsf{B}(G)}/\overline{\mathsf{B}(G)}^\circ\times\overline{\mathsf{B}(G)}/\overline{\mathsf{B}(G)}^\sharp$, which is torsion-free.

For (\ref{wy4}), first note that the equality $\overline{\mathsf{B}(G)}^\sharp=\overline{\mathsf{B}(G)}\cap G^\sharp$ is trivial. The inclusion $\overline{\mathsf{W}(G)}\subset G^\sharp\cap\overline{\mathsf{B}(G)}$ is clear (since $G^\sharp$ is closed). For the reverse inclusion, we can suppose that $\mathsf{W}(G^\circ)^\circ$ is trivial. So, by (\ref{wy3}), the topological direct product $\overline{\mathsf{B}(G)}^\circ\times\overline{\mathsf{B}(G)}^\sharp$ stands as an open subgroup in $\overline{\mathsf{B}(G)}$. Since the action by conjugation preserves this direct product decomposition, we deduce
\begin{equation}\label{eqpro}\overline{\mathsf{B}(G)}^\circ\times\overline{\mathsf{B}(G)}^\sharp=\overline{\mathsf{B}(G)\cap \overline{\mathsf{B}(G)}^\circ}\times\overline{\mathsf{B}(G)\cap \overline{\mathsf{B}(G)}^\sharp}.\end{equation}

So $\mathsf{B}(G)\cap \overline{\mathsf{B}(G)}^\sharp$ is dense in $\overline{\mathsf{B}(G)}^\sharp$. On the other hand, by (\ref{wy2}), it is equal to $\mathsf{W}(G)$. This shows that $\overline{\mathsf{W}(G)}$ is equal to $\overline{\mathsf{B}(G)}^\sharp$.


Let us show (\ref{wyz}). We can suppose that $\mathsf{W}(G^0)$ is trivial. We need the following fact: In a vector group standing as normal subgroup in a larger topological group, the set of elements with bounded conjugacy class is clearly a vector subspace, and hence is closed.

We deduce from (\ref{eqpro}) that $\mathsf{B}(G)\cap \overline{\mathsf{B}(G)}^\circ$ is dense in the vector group $\overline{\mathsf{B}(G)}^\circ$. Using the above fact, we deduce that 
$\mathsf{B}(G)\cap \overline{\mathsf{B}(G)}^\circ=\overline{\mathsf{B}(G)}^\circ$. So $\overline{\mathsf{B}(G)}^\circ\subset\mathsf{B}(G)$, and the equality immediately follows.

Let us show (\ref{wya}). Since $\overline{\mathsf{B}(G)}^\circ\overline{\mathsf{B}(G)}^\sharp$ is open (by (\ref{wy3})) in $\overline{\mathsf{B}(G)}$, we have $\overline{\mathsf{B}(G)}=\mathsf{B}(G)\overline{\mathsf{B}(G)}^\circ\overline{\mathsf{B}(G)}^\sharp=\mathsf{B}(G)\overline{\mathsf{B}(G)}^\sharp$, the latter equality holding since $\overline{\mathsf{B}(G)}^\circ$ is contained in $\mathsf{B}(G)$ (by (\ref{wyz})). To show that $\mathsf{B}(G)U=\mathsf{B}(G)\overline{\mathsf{B}(G)}^\sharp$, it is enough to check that $\overline{\mathsf{B}(G)}^\sharp\subset\mathsf{B}(G)U$, which holds because $\mathsf{B}(G)\cap \overline{\mathsf{B}(G)}^\sharp=\mathsf{W}(G)$ is dense in $\overline{\mathsf{B}(G)}^\sharp=\overline{\mathsf{W}(G)}$ (by (\ref{wy2}) and (\ref{wy4})). The last assertion immediately follows.

Let us show (\ref{wyb}). We can suppose that $\mathsf{W}(G^0)$ is trivial. So $\overline{\mathsf{B}(G)}^\sharp$ is totally disconnected; hence $\overline{\mathsf{B}(G)}^\sharp\cap G^0=\{1\}$. As we have checked in the proof of (\ref{wy3}), the projection $G\to G/G^\circ$ is proper in restriction to $\overline{\mathsf{B}(G)}^\sharp$. Hence $G^0$ and $\overline{\mathsf{B}(G)}^\sharp$ generate their topological direct product in $G$.
For $x\in G^0\cap\overline{\mathsf{B}(G)}$, since $\overline{\mathsf{B}(G)}=\mathsf{B}(G)\overline{\mathsf{B}(G)}^\sharp$ (by (\ref{wya})), we can write $x= yz$ with $y\in\mathsf{B}(G)$ and $z\in\overline{\mathsf{B}(G)}^\sharp$. In this topological direct product, $y$ decomposes as $y=xz^{-1}$. Since the projections of this decomposition stabilize $\mathsf{B}(G)$ (since they are continuous and $G$-equivariant), we deduce that $x\in \mathsf{B}(G)$. 
\end{proof}

\begin{remark} (See also Theorem \ref{wclo2}.) The subgroups $\mathsf{W}(G)$ and $\mathsf{B}(G)$ may be non-closed, as was initially observed by Tits \cite[Proposition 3]{Tits}; here we follow the similar \cite[\S 6, Example 1]{WY}: let $R=\mathbf{Z}/n\mathbf{Z}$ for some $n\ge 3$ and $R^\times$ its group of units and $G=R^{(\mathbf{N})}\rtimes (R^\times)^\mathbf{N}$; in this example $\mathsf{W}(G)=\mathsf{B}(G)\neq G$ is dense (here, $R^{(\mathbf{N})}$ is the direct sum of copies of $R$ indexed by $\mathbf{N}$).

We have $\mathsf{W}(\mathbf{Q}_p)=\mathbf{Q}_p$, which in this case is closed and noncompact. An example with $\mathsf{W}(G)$ closed and non-compact, in which $G$ is in addition compactly generated also exists, namely a suitable 4-dimensional $p$-adic Lie group whose center is isomorphic to $\mathbf{Q}_p$. (See also the examples in \S\ref{spc}.)
\end{remark}

\begin{lem}\label{wge}
--{\rm Now erased}\footnote{In the published version, Lemma \ref{wge} was a textual redundant copy of Proposition \ref{pwy}(\ref{wy2}).}--
\end{lem}

Proposition \ref{pwy}(\ref{wy2}) allows us to prove the following proposition, which is less obvious than it may seem at first sight.

\begin{prop}\label{ebgh}
Let $f:G\to H$ be a copci homomorphism between locally compact groups. Then $f^{-1}(\mathsf{W}(H))=\mathsf{W}(G)$.
\end{prop}
\begin{proof}
By properness, $f^{-1}(\mathsf{W}(H))\subset \mathsf{W}(G)$. So we have to prove the other inclusion, namely $f(\mathsf{W}(G))\subset \mathsf{W}(H)$. This amounts to showing that if $K$ is a compact normal subgroup of $G$, then $f(K)\subset \mathsf{W}(H)$. Indeed, writing $H=Mf(G)$ with $M$ a symmetric compact subset, we see that for every $h\in H$, $hf(K)h^{-1}\subset Mf(K)M$ and hence $f(K)\subset \mathsf{B}(H)$. Since $f(K)$ also consists of elliptic elements, by Proposition \ref{pwy}(\ref{wy2}), it follows that $f(K)\subset \mathsf{W}(H)$. 
\end{proof}

We now show the following theorem, which entails Theorem \ref{wclo}. In the case of
$\mathsf{B}(G)$ in totally disconnected groups, this result was already
noticed in \cite{Mol}, by observing that it follows from Trofimov's results and
Cayley-Abels graphs.

\begin{thm}\label{wclo2}If $G$ is a {\em compactly generated} locally compact group, then $\mathsf{B}(G)$ and $\mathsf{W}(G)$ are closed subgroups. In particular, there exists a compact normal subgroup $M$ of $G$ such that $\mathsf{W}(G/M)$ is discrete.
\end{thm}
\begin{proof}
We first check\footnote{In the published version, this was asserted without proof; these first eight lines have been added afterwards.} that if $G$ is a $\sigma$-compact locally compact group, then there is an ascending sequence $(W_n)$ of compact normal subgroups of $G$ whose union is $\mathsf{W}(G)$.

For $s\in G$, write $s^G=\{g^{-1}sg:g\in G\}$. Write $G=\bigcup M_n$ with $M_n$ a compact subset, $M_n$ contained in the interior of $M_{n+1}$ (so every compact subset of $G$ is contained in $M_n$ for some $n$). Write $K_n=\{s\in \mathsf{W}(G):s^G\subset M_n\}$. Then $K_n$ is a closed, conjugation-invariant subset of $G$, $K_n\subset K_{n+1}$ and $\bigcup K_n=\mathsf{W}(G)$. Then $K_n$ is a compact subset of $\mathsf{W}(G)$, which is contained in $G^\sharp$ by Proposition \ref{pwy}(\ref{wy4}). Hence the closure $W_n$ of the subgroup generated by $K_n$ is compact, $W_n\subset W_{n+1}$ and $\bigcup W_n=\mathsf{W}(G)$.


Let us next check the following independent claim: if $G$ is a $\sigma$-compact locally compact group such that $\mathsf{W}(G)$ is closed, then
\begin{enumerate}[(1)]
\item\label{iii1} any compact subset of $\mathsf{W}(G)$ is contained in a compact open subgroup of $\mathsf{W}(G)$ that is normal in $G$;
\item\label{iii2} there exists a compact normal subgroup $M$ of $G$ such that $\mathsf{W}(G/M)$ is discrete.
\end{enumerate}

To prove (\ref{iii1}), write, as above, $\mathsf{W}(G)$ as a union of an ascending sequence $(W_n)$ of compact normal subgroups of $G$. Since $\mathsf{W}(G)$ is closed, by the Baire theorem, $W_n$ is open in $\mathsf{W}(G)$ for all $n$ large enough. It follows that every compact subset of $\mathsf{W}(G)$ is contained in $W_n$ for $n$ large enough.


To prove (\ref{iii2}) let $K$ be the normal subgroup generated by some compact neighborhood of the identity in $\mathsf{W}(G)$. Then $K$ is compact (by (\ref{iii1})) and open in $\mathsf{W}(G)$. So $\mathsf{W}(G/K)=\mathsf{W}(G)/K$ is discrete in $G/K$. Hence, the first assertion of the theorem implies the second one.

Let us now prove the first assertion, namely that $\mathsf{B}(G)$ and $\mathsf{W}(G)$ are closed if $G$ is compactly generated. By Proposition \ref{pwy}(\ref{wy2}), it is enough to deal with $\mathsf{B}(G)$. We first check it when $G$ is totally disconnected.

Trofimov \cite{Tro} deals with the case of the automorphism group of a vertex-transitive finite valency connected graph (this is the hard case). Therefore, if $G$ is a locally compact group with a proper transitive action $j:G\to\Aut(X)$ on such a graph $X$, then $\mathsf{B}(G)=j^{-1}(\mathsf{B}(\Aut(X))$ is closed as well. Abels \cite[Beispiel~5.2]{Abe} observed that any compactly generated totally disconnected locally compact group admits such an action (now known as Cayley-Abels graph).

So the first statement is proved when $G$ is totally disconnected (as already
noticed in \cite{Mol}). We conclude that $\mathsf{W}(G)$ is closed by proving the following claim\footnote{The published proof is based on a similar (correct) claim, but its proof contains a gap, as it relied on a wrong assertion of late Proposition \ref{pwy}(4).}: {\it for every locally compact group $G$, if $\mathsf{W}(G/G^\circ)$ is closed in $G/G^\circ$, then $\mathsf{W}(G)$ is closed in $G$}.


Let $G$ be as in the claim. As we have already checked, there exists a compact normal subgroup of $G/G^\circ$ that is open in $\mathsf{W}(G/G^\circ)$. Let $V$ be its inverse image in $G$. So $V$ is normal in $G$ and $\mathsf{W}(G/V)$ is discrete. 

It follows that $\overline{\mathsf{W}(G)}\cap V$ is open in $\overline{\mathsf{W}(G)}$. It is also contained in $V^\sharp$; since $V$ is connected-by-compact, $V^\sharp$ is compact, and hence contained in $\mathsf{W}(G)$. Thus $\mathsf{W}(G)$ is an open subgroup of its closure, and hence is closed.
 
To conclude that $\mathsf{B}(G)$ is also closed, we apply Proposition \ref{pwy}(\ref{wya}).
\end{proof}

\begin{remark}
If $f:G\to H$ is a copci homomorphism and $H$ has a compact polycompact radical, then so does $G$, as a consequence of Proposition \ref{ebgh}.

Nevertheless, to have a compact polycompact radical is not stable under commability. Indeed, consider the infinite alternating group $\textnormal{Alt}_\infty$ and the direct sum $\bigoplus_{n\ge 2}\mathbf{Z}/n\mathbf{Z}$; the former is W-faithful while the latter is equal to its polyfinite radical. By \cite[Lemma 5.1]{BDHM}, $\Gamma$ and $\Lambda$ possess isometric left-invariant proper distances; thus if $X$ is the underlying metric space, then both groups are cocompact in $\Isom(X)$.

More interestingly, there exist {\em finitely generated} groups that are quasi-isometric and actually commable (and cocompact in the same locally compact group), one with finite and the other with infinite polyfinite radical; see \S\ref{spc}.
\end{remark}

\begin{lem}\label{wssc}
Let $G,H$ be commable groups. Assume that $\mathsf{W}(G)=\mathsf{W}(H)=\{1\}$ and that every LC-group commable to $G$ and $H$ has a compact polycompact radical. Then $G$ and $H$ are strictly commable.
\end{lem}
\begin{proof}
It follows from Proposition \ref{ebgh} that any copci homomorphism $A\to B$ between LC-groups with compact polycompact radicals factors through an injective copci homomorphism $A/\mathsf{W}(A)\to B/\mathsf{W}(B)$. The result immediately follows.
\end{proof}

\begin{remark}
The condition ``every LC-group commable to $G$ and $H$ has a compact polycompact radical" is hard to check in general. However, it is true in the following three cases:
\begin{itemize}
\item hyperbolic locally compact groups;
\item locally compact groups with at least two ends;
\item locally compact groups of polynomial growth.
\end{itemize}
\end{remark}

Let us now go back to Definition \ref{d_commable}. For the sake of readability, it is convenient to use arrows such as 
$\nearrow$ and $\nwarrow$ to denote copci homomorphisms. For instance, we say that $G,H$ are commable through $\nearrow\nwarrow\nearrow$ if there exist LC-groups $G_1,G_2$ and copci homomorphisms $G\nearrow G_1\nwarrow G_2\nearrow H$.

\begin{example}
Let $\Gamma$ be a cocompact subgroup of $\SL_2(\mathbf{R})$ and $\widetilde{\Gamma}$ its inverse image in $\widetilde{\SL}_2(\mathbf{R})$. Then $\widetilde{\Gamma}$ and $\Gamma\times\mathbf{Z}$ are commable through $\nearrow\nwarrow\nearrow\nwarrow$.

Namely, denoting by $T$ the upper triangular group in $\SL_2(\mathbf{R})$ with positive diagonal entries, there are (injective) copci homomorphisms \[\widetilde{\Gamma}\to \widetilde{\SL}_2(\mathbf{R})\leftarrow T\times\mathbf{Z}\to \SL_2(\mathbf{R})\times\mathbf{Z}\leftarrow \Gamma\times\mathbf{Z}.\]
\end{example}

\subsection{An example for polyfinite radicals}\label{spc}

We show here that among compactly generated locally compact groups, to have compact polycompact radical is not a commability invariant (and hence not a quasi-isometry invariant), and actually exhibit two finitely generated groups that are cocompact lattices in the same locally compact group, one of which has a trivial polyfinite radical, and the other one having an infinite one.

If $R$ is commutative ring (associative with unit), let $H_3(R)$ denote the Heisenberg group over $R$, namely those upper triangular $3\times 3$-matrices with $1$ on the diagonal. If $t,u\in R^\times$ are invertible elements, define the automorphism $\varphi(u,v,uv)$ of $H_3(R)$ defined by the assignment
\[\begin{pmatrix}1 & x & z\\0 & 1 & y\\0 & 0 & 1\end{pmatrix}\mapsto \begin{pmatrix}1 & ux & uvz\\0 & 1 & vy\\0 & 0 & 1\end{pmatrix}.\]
We also define $Z(R)$ as the subgroup of $H_3(R)$ consisting of the matrices such as above with $(x,y)=(0,0)$; it is the center of $H_3(R)$.

We henceforth assume that $R$ is a commutative $\mathbf{F}_p[t,t^{-1}]$-algebra ($t$ being an indeterminate), and we denote
\[s=t^2+t^{-2},\quad u_1=1+t,\quad u_2=1+t^{-1};\]
we also define, for $s,u_1,u_2$ in $R^\times$, the elements
\[\alpha_1=\varphi(s,s^{-1}u_2,u_2),\;\beta_1=\varphi(s,s^{-1}u_1,u_1),\;\alpha_2=\varphi(s,s^{-1}u_2u_1,u_2u_1),\;\beta_2=\varphi(s,s^{-1},1).\] 
and, for $i=1,2$
\[\Gamma_i=H_3(\mathbf{F}_p[t,t^{-1}])^2\rtimes_{(\alpha_i,\beta_i)}\mathbf{Z},\]
where the notation means that the action is given by $n\cdot (A,B)=(\alpha_i^n(A),\beta_i^n(B))$, for $(A,B)\in H_3(\mathbf{F}_p[t,t^{-1}])^2$.

\begin{thm}
The polyfinite radical $\mathsf{W}(\Gamma_1)$ is trivial, while $\mathsf{W}(\Gamma_2)$ is infinite; $\Gamma_1$ and $\Gamma_2$ are finitely generated and are isomorphic to cocompact lattices in the same locally compact group.
\end{thm}

\begin{remark}
The proof will also show that
\begin{itemize}
\item the group $\Gamma_1$ is {\em icc}, that is, the only finite conjugacy class is $\{1\}$, which also means that its FC-center $\mathsf{B}(G)$ (Definition \ref{dpr}) is trivial, while
\item the group $\Gamma_2$ has an infinite (locally finite) center.
\end{itemize}
\end{remark}

Let us proceed to the proof, which is not particularly hard, part of the difficulty being to exhibit the group itself. 

The easiest part is that $\mathsf{W}(\Gamma_2)$ is infinite: indeed, it contains (and is actually equal to) the subgroup $\{1\}\times Z(\mathbf{F}_p[t,t^{-1}])\subset H_3(\mathbf{F}_p[t,t^{-1}])^2$, which is an infinite locally finite central subgroup (it is central in $\Gamma_2$ because the action of $\beta_2$ on the center is trivial). 

The polyfinite radical $\mathsf{W}(\Gamma_1)$ is trivial; actually every nontrivial conjugacy class in $\Gamma_1$ is infinite, i.e., the FC-center $\mathsf{B}(\Gamma_1)$ is trivial (see Definition \ref{dpr}). To see this, argue as follows: since the centralizer of $H_3(\mathbf{F}_p[t,t^{-1}])^2$ in $\Gamma_1$ is itself, if we had $\mathsf{B}(\Gamma_1)$ nontrivial, the intersection $\mathsf{B}(\Gamma_1)\cap H_3(\mathbf{F}_p[t,t^{-1}])^2$ would be nontrivial as well, but it is clear from the given action, and more precisely from the fact that none of $s$, $s^{-1}u_2$, $u_2$, $s^{-1}u_1$, $u_1$ is a root of unity, that for every nontrivial element of $H_3(\mathbf{F}_p[t,t^{-1}])^2$, the set of its conjugates is infinite.

Now let us come to the embedding. We first embed each $\Gamma_i$ in a locally compact group $G_i$ defined as
\[G_i=H_3\big(\mathbf{F}_p\lp t\rp\big)^2\times H_3\big(\mathbf{F}_p\lp t^{-1}\rp\big)^2\rtimes_{(\alpha_i,\beta_i)}\mathbf{Z}.\]
The natural discrete embedding of $\Gamma_i$ into $G_i$ is induced by the ring diagonal embedding of $\mathbf{F}_p[t,t^{-1}]$ into $\mathbf{F}_p\lp t\rp\times \mathbf{F}_p\lp t^{-1}\rp$, and is cocompact for standard reasons. 

That each $G_i$ is compactly generated is easily checked, since we can go up and down in the $x$ and $y$ directions, using that none of the elements $s^{\pm 1}$, $s^{-1}u_i$, $s^{-1}u_1u_2$ has modulus 1, and get the $z$-direction by performing commutators (the routine argument is left to the reader). It follows that the cocompact lattices $\Gamma_i$ are finitely generated.

Let us now describe each $G_i$ in a more convenient way, using the obvious ``involutive" isomorphism $\mathbf{F}_p\lp t\rp\simeq \mathbf{F}_p\lp t^{-1}\rp$, which exchanges $u_1$ and $u_2$ and ``fixes" $s$: we have 
\[G_i=H_3\big(\mathbf{F}_p\lp t\rp\big)^4\rtimes_{(\alpha_i,\beta_i,\gamma_i,\delta_i)}\mathbf{Z},\]
where
\[\gamma_1=\beta_1,\;\delta_1=\alpha_1,\;\gamma_2=\alpha_2,\delta_2=\beta_2.\]

Let $P$ be the 1-sphere in $\mathbf{F}_p\lp t\rp$; this is a compact group under multiplication; note that $u_1\in P$. Also define, for $v\in P$, $\psi(v)=\varphi(1,v,v)$.
Let us embed $G_i$ cocompactly in a larger group defined as
\[M_i\;=\;  G_i\rtimes P^4 \;=\; H_3\big(\mathbf{F}_p\lp t\rp\big)^4\rtimes \Big( \mathbf{Z}\times P^4\Big)\]
here the action of $\mathbf{Z}$ is still given by $(\alpha_i,\beta_i,\delta_i,\gamma_i)$,  while the action of $(v_1,v_2,v_3,v_4)\in P^4$ on $H_3\big(\mathbf{F}_p\lp t\rp\big)^4$ is given on the $k$-th component ($1\le k\le 4$) by $\psi(v_k)$. Note that this is a well-defined action of the abelian group $\mathbf{Z}\times P^4$, because the two action commute.

The conclusion of the proof is given by the following claim: $M_1$ and $M_2$ are isomorphic as topological groups. This is not hard to see: indeed, the automorphisms $(\alpha_1,\beta_1,\gamma_1,\delta_1)$ and $(\alpha_2,\beta_2,\delta_2,\gamma_2)$ of $H_3\big(\mathbf{F}_p\lp t\rp\big)^4$ differ by an element of $P^4$, namely the element $(\psi(u_1^{-1}),\psi(u_1),\penalty0\psi(u_1),\psi(u_1^{-1}))$. Thus both groups of automorphisms isomorphic to $Z\times P^4$ are equal (as subgroups of automorphisms of $H_3\big(\mathbf{F}_p\lp t\rp\big)^4$) and $M_1$ and $M_2$ are isomorphic (note that the isomorphism thus involves switching the 3rd and 4th components, as we have switched $\delta_2$ and $\gamma_2$).

\section{Generalities on focal groups}

Let us set up some generalities about focal group. We provide, insofar as possible, proofs that avoid using hyperbolicity or large-scale geometry arguments.

When an LC-group $G$ is focal, we usually write $G=N\rtimes\Lambda$ as in Definition \ref{defocal}, where $\Lambda\in\{\mathbf{Z},\mathbf{R}\}$ compacts $N$.

\begin{lem}\label{fow}
Every focal LC-group $G$ has a compact polycompact radical.

If $H$ is an LC-group and $W$ a compact normal subgroup, then $H$ is focal if and only if $H/W$ is focal.
\end{lem}
\begin{proof}
Clearly $\mathsf{W}(G)\subset N$. If $\Omega$ is a compact vacuum subset for a compacting element, then clearly every compact subgroup of $G$ contained in $N$ is contained in $\Omega$.

In the second statement, both verifications are straightforward.
\end{proof}

\begin{lem}
The subgroup $N$ is the kernel of the modular function on $G$. If $G$ is not of connected type, then $\Lambda=\mathbf{Z}$.
\end{lem}
\begin{proof}
If $\Delta$ is the modular function, since $N$ is compacted, the closure of $\Delta(G)$ is compact, and therefore is trivial. So $\Ker(\Delta)$ contains $N$; hence it is either $N$ or $G$, but the latter case is excluded since $N$ is noncompact.

If by contradiction $G$ is of totally disconnected type and $\Lambda=\mathbf{R}$, then modding out by $\mathsf{W}(G)$ (see Lemma \ref{fow}), we can suppose that $\mathsf{W}(G)$ is trivial, so $N$ is totally disconnected and hence centralizes $\Lambda$. This contradicts that $N$ is noncompact and compacted by $\Lambda$.
If $G$ is not of connected type, it follows from the definition that $G/N^\circ=(N/N^\circ)\rtimes\Lambda$ is focal of totally disconnected type, so by the previous case we get $\Lambda=\mathbf{Z}$.
\end{proof}

\begin{prop}[{\cite[Theorem A.5]{CoTe}}]\label{compde}
Let $N$ be an LC-group such that $\mathsf{W}(N^\circ)=\{1\}$. Suppose that $N$ has a compacting automorphism. Then the subgroups $N^\circ$ and $N^\sharp$ (the elliptic radical) generate their topological direct product, which is open of finite index in $N$.
\end{prop}

\begin{lem}\label{focoty}
Let $G=N\rtimes\Lambda$ be a focal LC-group. Denote by $\pi$ the projection $G\to\Lambda$. 
Then for every closed cocompact subgroup $H\subset G$, the intersection $H\cap N$ is cocompact in $N$, the projection $\pi(H)$ is closed (and cocompact) in $\Lambda$, and $H$ is focal of the same type as $G$. 
\end{lem}
\begin{proof}
Let us first show that $\pi(H)$ is closed (note that this is trivial if $\Lambda=\mathbf{Z}$). If $(h_n)$ is in $H$ and $\pi(h_n)$ converges, write $h_n=\nu_n\pi(h_n)$ with $\nu_n\in N$. Conjugating by an element of $\Lambda$, we can arrange $\nu_n$ (and hence $h_n$) to be bounded without modifying $\pi(h_n)$. 

We freely use that the set of compactions of $N$ is stable under multiplication by inner automorphisms \cite[Lemma 6.16]{CCMT}.

If $\Lambda=\mathbf{Z}$ or more generally if the projection of $H$ in $\Lambda$ has a discrete image, then $(H\cap N)\rtimes \langle\sigma\rangle$ is open of finite index in $H$ for any $\sigma\in H\smallsetminus\Ker(\Delta_G)$; it is focal.

If $\Lambda=\mathbf{R}$, we discuss. If $H^\circ\nsubseteq N$, then $H$ has a 1-parameter subgroup $P$ on which $\log\circ\Delta_G$ is an isomorphism onto $\mathbf{R}$. Thus $H\cap N$ is cocompact in $N$ and $H=(H\cap N)\rtimes P$ is focal. 

If $\Lambda=\mathbf{R}$ and $H^\circ\subset N$, then the image of $N$ in $\Lambda=\mathbf{R}$ is countable, and since it is also closed by the beginning of the proof, it is discrete and this case was covered.

It remains to check that $H$ has the same type as $G$. In view of Lemma \ref{fow}, we can suppose that $\mathsf{W}(G)=\{1\}$.
\begin{itemize}
\item If $G$ is of totally disconnected type, then $H$ is totally disconnected, hence of totally disconnected type.
\item If $G$ is of connected type, it is a Lie group, and hence $H$ is also a Lie group. So $N\cap H$ is a Lie group with a compacting automorphism, hence it has finitely many components. So $H$ is of connected type as well.
\item If $H$ is of connected type, then $N\cap H$ is virtually connected. Since it is cocompact in $N$, it follows that $N$ is connected-by-compact, hence virtually connected (since $\mathsf{W}(N)=\{1\}$), so $G$ is of connected type as well.
\item If $H$ is of totally disconnected type, we use that $N$ has a characteristic open subgroup of finite index $N'$ with a characteristic direct decomposition $E\times M$, with $E$ totally disconnected and $M$ connected (see Proposition \ref{compde}). So the projection of $H\cap N'$ on $E$ is finite. Since it is cocompact, it follows that $M$ is compact, and thus $G$ has totally disconnected type. 
\end{itemize}
\end{proof}

\begin{remark}
The last statement in Lemma \ref{focoty} can also be proved using geometric arguments based on hyperbolicity and the topology of the boundary (see Proposition \ref{bdmi}): indeed the boundary is a Cantor space exactly in the totally disconnected type and is a topological sphere exactly in the connected type.
\end{remark}

\begin{lem}\label{l_sg}
Let $G$ be a focal LC-group. There exists a unique integer $s_G\ge 1$ such that the image $I_G$ of $\Delta_{G/G^\circ}$ is the multiplicative cyclic subgroup of $\mathbf{Q}^*$ generated by $s_G$. We have $s_G=1$ if and only if $G$ is of connected type.
\end{lem}
\begin{proof}
The uniqueness is clear.

If $G$ has connected type, then $G/G^\circ$ is either compact or compact-by-$\mathbf{Z}$, so $G/G^\circ$ is unimodular and we have $s_G=1$.

Otherwise, $G/G^\circ$ is focal of totally disconnected type, and hence we can reduce to when $G=N\rtimes_\sigma\mathbf{Z}$ is focal and totally disconnected. 
Then if $K$ is an open compact subgroup of $N$ such that $\sigma(K)\subset K$, the image of the modular map is the multiplicative subgroup generated by the index $s_G=[K:\sigma(K)]$. We have $s_G\ge 2$, since otherwise $K$ would have finite index in $N$, which would be compact, contradicting the definition of focal. 
\end{proof}

\begin{lem}\label{incsha}
Let $f:H\to G$ be a continuous homomorphism between focal LC-groups with cocompact image. Then $f(H^\sharp)\subset G^\sharp$.
\end{lem}
\begin{proof}
Composing with the projection, we can suppose that $G^\sharp=\{1\}$. Thus $G$ is a W-faithful Lie group, $\Ker(\Delta_G)$ is a virtually connected Lie group, and we need to show that $H^\sharp\subset\Ker(f)$. Since $K=\overline{f(H^\sharp)}$ is elliptic, it is contained in $\Ker(\Delta_G)$ and hence is compact. Let $G'$ be the normalizer of $K$. Then $G'$ contains $f(H)$ which is cocompact, and is closed; we have $K\subset \mathsf{W}(G')$. On the other hand, by Proposition \ref{ebgh}, we have $\mathsf{W}(G')\subset \mathsf{W}(G)=\{1\}$. Hence $K=\{1\}$. 
\end{proof}

\begin{lem}\label{facgo}
Let $f:H\to G$ be a copci homomorphism between focal LC-groups. Then
\begin{enumerate}[(1)]
\item\label{i_ke} $f(H)\cap\Ker(\Delta_G)$ is cocompact in $\Ker(\Delta_G)$;
\item\label{i_noco} if $G,H$ are not of connected type,
the factored map $\bar{f}:H/H^\circ\to G/G^\circ$ is a copci homomorphism between focal LC-groups of totally disconnected type.
\item\label{i_notd} if $G,H$ are not of totally disconnected type,
the factored map $\hat{f}:H/H^\sharp\to G/G^\sharp$ is a copci homomorphism between focal LC-groups of totally disconnected type.
\end{enumerate}
\end{lem}
\begin{proof}
For (\ref{i_ke}), this follows from the fact that $\Delta_G(f(H))=\Delta_H(H)$ is closed.

For (\ref{i_noco}), we have to check is that $\bar{f}$ is copci. Since the LC-groups involved are $\sigma$-compact, we have to check that $\bar{f}$ has a compact kernel and a closed cocompact image. That the image is cocompact is clear.

We can suppose that $\mathsf{W}(H)=\mathsf{W}(G)=\{1\}$ (using Proposition \ref{ebgh}), and thus $f$ is injective and $G^\circ$ is a Lie group. In particular, any closed totally disconnected subgroup of $f^{-1}(G^\circ)$ is discrete.

Recall (Proposition \ref{compde}) that $\Ker(\Delta_H)$ has an open characteristic subgroup of finite index $U_H$ with a characteristic decomposition $H^\circ\times H^\sharp$ with $H^\sharp$ totally disconnected. So $f^{-1}(G^\circ)=H^\circ\times (H^\sharp\cap f^{-1}(G^\circ))$. We deduce from the previous lines that $H^\sharp\cap f^{-1}(G^\circ)$ is discrete; since it admits a compacting automorphism, it is finite. Since $\mathsf{W}(H)=\{1\}$, it is thus trivial, and we deduce that $f^{-1}(G^\circ)=H^\circ$. So $\bar{f}$ is injective.

Finally, to check that $\bar{f}$ has a closed image amounts to showing that $G^\circ f(H)$ is closed; to check this it is enough to show that $f(H)\cap G^\circ=f(H^\circ)$ is cocompact in $G^\circ$. 

By Lemma \ref{incsha}, we have $f(H^\sharp)\subset G^\sharp$.
 Since $f(H^\sharp\times H^\circ)$ is cocompact in $G^\sharp\times G^\circ$, we deduce that $f(H^\circ)$ is cocompact in $G^\circ$.

The statement (\ref{i_notd}) is proved along the same lines.
\end{proof}

\section{The $q$-invariant and commability classification in totally disconnected type}

\subsection{The $q$-invariant}

We say that a positive integer is a non-power if it is not an integral power of a smaller positive integer (thus non-power integers are $1,2,3,5,6,7,10,\dots$). Each integer $n\ge 1$ is an integral power of a unique non-power integer denoted by $\sqrt[\max]{n}$.

\begin{defn}
Let $G$ be a focal LC-group, and let $s_G$ be the generator of the image of $\Delta_{G/G^\circ}$ as in Lemma \ref{l_sg}. We define $q_G=\sqrt[\max]{s_G}$ and call it the q-invariant of $G$.
\end{defn}

\begin{prop}\label{p_q}
Let $G$ be a focal LC-group. Then $q_G$ is the unique non-power integer such that the image of $\Delta_{G/G^\circ}$ has finite index in the multiplicative group $\{q_G^n:n\in\mathbf{Z}\}$. We have $q_G=1$ if and only $G$ has connected type. The number $q_G$ is an invariant of the commability class of $G$ within focal groups.
\end{prop}
\begin{proof}
The first two statement immediately follow from the definition and Lemma \ref{l_sg}. Let us prove the last one. Let $f:H\to G$ be a copci homomorphism between focal groups and let us show $q_H=q_G$. 

Recall that $H$ and $G$ have the same type (Lemma \ref{focoty}), and in connected type we have $q_H=q_G=1$. 
 
Assume that $G$ and $H$ have totally disconnected type.
Since $G$ is amenable, we have $\Delta_H=\Delta_G\circ f$ \cite[Corollary B.1.7]{BHV}. It follows that $\Delta_H(H)$ is cocompact in $\Delta_G(G)$; since these are cyclic groups, it has finite index, say $k$, and therefore we obtain that $s_H=s_G^k$. So $q_H=q_G$.

If $G$ and $H$ have mixed type, then $\bar{f}:H/H^\circ\to G/G^\circ$ is copci by Lemma \ref{facgo}, and since $q_G=q_{G/G^\circ}$ and $q_H=q_{H/H^\circ}$ we can reduce to the totally disconnected type.
\end{proof}

Define, for every $m\ge 2$, the group $\FT_m$ as the stabilizer of a boundary point in the automorphism group of an $(m+1)$-regular tree.

We begin with the standard observation:
\begin{lem}\label{ftrm}
The group $\FT_m$ is focal of totally disconnected type, and its q-invariant is $\sqrt[\max]{m}$ (which in particular is $m$ if $m$ is a non-power integer).
\end{lem}
\begin{proof}
Let $T$ be the $(m+1)$-regular tree, $\omega$ a boundary point and $x_0$ a fixed vertex, and write $G=\FT_m$ as the stabilizer of $\omega$ in $\Aut(T)$. Consider the geodesic ray $(x_0,x_1,\dots)$ ending at $\omega$.
Let $G_{x_0}$ be the stabilizer of $x_0$, this is a compact open subgroup (actually fixing all $x_n$ for $n\ge 0$). Extend this ray to a bi-infinite geodesic $(x_n)_{n\in\mathbf{Z}}$, and let $h$ be a loxodromic isometry along this axis, mapping $x_n$ to $x_{n-1}$ for all $n\in \mathbf{Z}$. Let $N$ be the subgroup of those elements in $G$ fixing $x_n$ for large enough $n$, it is not compact since $m+1\ge 3$. Then $G=N\rtimes\langle h\rangle$, and $hG_{x_0}h^{-1}=G_{x_{-1}}\subset G_{x_0}$ and $N=\bigcup_{n\ge 0}h^{-n}G_{x_0}h^n$, thus $h$ acts on $N$ as a compacting automorphism and thus $G$ is focal. We have $[G_{x_0}:G_{x_{-1}}]=m$ and thus $s_G=m$, so $q_G=\sqrt[\max]{m}$.
\end{proof}

\begin{prop}\label{copfo}
Let $G$ be an LC-group and $q\ge 2$ a non-power integer. Then the following are equivalent:
\begin{enumerate}[(i)]
\item\label{i_coft} $G$ admits a copci homomorphism into $\FT_{q^k}$ for some integer $k\ge 1$;
\item\label{i_foq} $G$ is focal of totally disconnected type and $q_G=q$.
\end{enumerate}
More generally, the following are equivalent, if $m\ge 2$ is any integer.
\begin{enumerate}[(i')]
\item\label{i_coft2} $G$ admits a copci homomorphism into $\FT_{m^k}$ for some integer $k\ge 1$;
\item\label{i_foq2} $G$ is focal of totally disconnected type and $s_G$ is a power of $m$.
\end{enumerate}

\end{prop}
\begin{proof}
Suppose (\ref{i_coft}). By Lemma \ref{ftrm}, $\FT_{q^k}$ is focal of totally disconnected type and its q-invariant is $q$. Hence $G$ is also focal of totally disconnected type, and $q_G=q$ by Proposition \ref{p_q}.

Conversely, suppose (\ref{i_foq}). Write $G=N\rtimes_\sigma\mathbf{Z}$, and let $K$ be a compact open subgroup of $G$ such that $\sigma(K)\subset K$ and $\bigcup_{n\ge 0}\sigma^{-n}(K)=N$. We have $[K:\sigma(K)]=s_G=q^k$ for some $k$. Thus $G$ is the ascending HNN-extension of $K$ endowed with its endomorphism $\sigma$. The Bass-Serre tree is $(q^k+1)$-regular, and $G$ fixes a boundary point; this action is proper and vertex-transitive, so it defines a copci homomorphism $G\to\FT_{q^k}$. (This construction is already used in \cite[Lemma 6.7]{CCMT}.)

The more general equivalence is proved by exactly the same argument.
\end{proof}

If $G$ is a focal LC-group of the form $N\rtimes\mathbf{Z}$ (e.g., not of connected type), and $n\ge 1$, define $G^{[n]}$ as the open normal subgroup $N\rtimes(n\mathbf{Z})$. It is the unique normal subgroup of $G$ containing $\Ker(\Delta_G)$ such that the quotient is cyclic of order $n$.

\begin{cor}\label{tdckfo2}
Let $G_1,G_2$ be focal LC-groups of totally disconnected type. The following are equivalent:
\begin{enumerate}
\item\label{fccck1} $G_1$ and $G_2$ are commable within focal groups;
\item\label{fccck4} $q_{G_1}=q_{G_2}$ (i.e., $\Delta(G_1)$ and $\Delta(G_2)$ are commensurate subgroups of $\mathbf{Q}^*$)
\item\label{fccck2} there is a commation within focal groups $G_1\nearrow\nwarrow\nearrow\nwarrow G_2$;
\item\label{fccck8} there exists a non-power integer $q\ge 2$ and an integer $n\ge 1$, such that for $i=1,2$ there is a commation within focal groups $G_i\nearrow\nwarrow \FT_{q}^{[n]}$;
\item\label{fccck6} there exists a non-power integer $q\ge 2$ such that for $i=1,2$, there is a commation within focal groups with $G_i\nearrow\nwarrow\nearrow \FT_q$;
\item\label{fccck3} there is a commation within focal groups $G_1\nwarrow\nearrow\nwarrow\nearrow G_2$;
\item\label{fccck7} there exists an integer $m\ge 2$, such that for $i=1,2$ there is a commation within focal groups $G_i\nwarrow\nearrow \FT_m$.
\end{enumerate}
\end{cor}

\begin{proof}
Obviously any of the conditions (\ref{fccck2}),  (\ref{fccck8}), (\ref{fccck6}), (\ref{fccck3}), (\ref{fccck7}) implies (\ref{fccck1}).

By Proposition \ref{p_q}, (\ref{fccck1}) implies (\ref{fccck4}). 

Some of the remaining implications are immediate:
\begin{itemize}
\item
Composing with the inclusion $\FT_{q}^{[n]}\to \FT_q$, we obtain the implication (\ref{fccck8})$\Rightarrow$(\ref{fccck6}).
\item (\ref{fccck7})$\Rightarrow$(\ref{fccck3}) and (\ref{fccck8})$\Rightarrow$(\ref{fccck2}) are clear.
\end{itemize}

Assume (\ref{fccck4}). Then by Proposition \ref{copfo}, there are copci homomorphisms $G_i\to \FT_{q^{n_i}}$ for some $n_1,n_2$ (which can be chosen as $q^{n_i}=s_{G_i}$). 
If $n=\max(n_1,n_2)$, then there is (by the second equivalence in Proposition \ref{copfo}) a copci homomorphism $\FT_q^{[n_i]}\to \FT_{q^{n_i}}$, which induces a copci homomorphism $\FT_q^{[n]}\to \FT_{q^{n_i}}^{[n-n_i]}\subset \FT_{q^{n_i}}$ and thus we have copci homomorphisms $G_i\to \FT_{q^{n_i}}\leftarrow \FT_q^{[n]}$, so (\ref{fccck8}) holds.

Also assuming (\ref{fccck4}), taking the same $n$, we have copci homomorphisms $G_i\leftarrow G_i^{[n-n_i]}\to \FT_{q^n}$, proving (\ref{fccck7}). 
\end{proof}

\begin{lem}\label{copft}
For any two distinct integers $m,n\ge 2$, there is no copci homomorphism $f:\FT_n\to\FT_m$.
\end{lem}
\begin{proof}
We can suppose $m<n$. Considering the ranges of the modular function, the condition implies that $n$ is an integral power of $m$. In particular $n\ge m^2\ge 2m$.

By Chebyshev's Theorem (Bertrand's postulate), there exists a prime $p$ in $\mathopen]m,2m\mathclose[$. Thus $m<p\le n$.

Since $\mathsf{W}(\FT_n)=\{1\}$, $f$ should be injective. We conclude by the observation that $\FT_n$ contains an element of order $p$, while $\FT_m$ has none.
\end{proof}

A strengthening of Lemma \ref{copft} will be given in Proposition \ref{gimp}.

\begin{remark}\label{rsous}
For any integers $\ell\ge 2$, $m,n\ge 1$, the groups $\FT_{\ell^m}$ and $\FT_{\ell^n}$ are commable through $\nwarrow\nearrow$. Indeed, both contain $\mathbf{Q}_\ell\rtimes_{\ell^{nm}}\mathbf{Z}$ as a closed cocompact subgroup.
\end{remark}

\begin{que}
If $G_1$ and $G_2$ are focal LC-groups of totally disconnected type and are commable, are they always commable through $G_1\nwarrow\nearrow\nwarrow G_2$?
\end{que}

\section{Commability in connected type}\label{sct}

We now study commability between focal LC-groups of connected type; this is technically the most involved case.

We will prove the following theorem (see Corollary \ref{efou}).

\begin{thm}\label{classcomcon}
Let $G$ be a focal LC-group of connected type. Then there exists a W-faithful focal LC-group $\hat{G}$ satisfying the following: for every LC-group $H$, the group $H$ is commable to $G$ within focal groups if and only if there is a copci homomorphism $H\to \hat{G}$. Moreover, $\hat{G}$ is unique up to topological isomorphism.
\end{thm}

Thanks to Lemma \ref{wssc}, most of the study of commability for focal LC-groups can be carried out in the context of W-faithful focal LC-groups, and therefore the W-faithfulness assumption will often be done in the next lemmas.

Let $N$ be a simply connected nilpotent Lie group and $\mk{n}$ its Lie algebra. The tangent map at the identity induces a canonical isomorphism $\Aut(N)\to\Aut(\mk{n})$. We say that $\alpha\in\Aut(N)$ is purely positive real if all its eigenvalues are real, and that a subgroup of $\Aut(N)$ is purely positive real if all its elements are purely positive real. Note that the product of two commuting purely positive real elements is purely positive real.

Recall that a topological group is called a vector group if it is isomorphic to $\mathbf{R}^k$ for some $k\ge 0$.

\begin{lem}\label{l_par}
Let $\mathbb{G}$ be a real linear algebraic group and $G=\mathbb{G}(\mathbf{R})$. For any purely real subgroup $A\subset G$, closed and connected in the Lie topology, there exists a unique minimal vector group $\Par(A)\subset\Aut(N)$ containing $A$ and contained in the Zariski closure of $A$; the inclusion of $A$ in $\Par(A)$ is cocompact. 
\end{lem}
\begin{proof}
We can fix a real embedding of $\mathbb{G}$ into $\GL_n$ and thus interpret $G$ as a closed subgroup of $\GL_n(\mathbf{R})$.

The Zariski closure $B$ of $A$ consists of matrices with only real eigenvalues. In particular, its identity component $B^\circ$ in the ordinary topology consists of matrices with only positive real eigenvalues. It follows that $B^\circ$ is a connected abelian Lie group and has no nontrivial compact subgroup and therefore $B^\circ$ is a vector group. The existence and uniqueness result is now clear, namely $\Par(A)$ corresponds to the vector subspace generated by $A$; cocompactness follows.
\end{proof}

\begin{lem}\label{dec_ri}
Every automorphism $\alpha\in\Aut(N)$ can be written in a unique way as $\alpha=\alpha_{\textnormal{r}}\alpha_{\textnormal{i}}$ where $\alpha_{\textnormal{r}}$ is purely positive real, $\alpha_{\textnormal{i}}$ is elliptic (belongs to a compact subgroup of $\Aut(N)$) and $[\alpha_{\textnormal{r}},\alpha_{\textnormal{i}}]=1$. Moreover, $\alpha_{\textnormal{r}}$ and $\alpha_{\textnormal{i}}$ belong to the Zariski closure of the subgroup generated by $\alpha$; in particular the centralizer of $\alpha_{\textnormal{d}}$ and $\alpha_{\textnormal{i}}$ each contain the centralizer of $\alpha$.
\end{lem}
\begin{proof}[Proof (sketched)]
We work in the Zariski closure of the subgroup generated by $\alpha$, which is abelian. We have a (unique) similar decomposition $\alpha=\alpha_{\textnormal{s}}\alpha_{\textnormal{u}}$ where $\alpha_{\textnormal{s}}$ is $\mathbf{C}$-diagonalizable and $\alpha_{\textnormal{u}}$ is unipotent, and a (non-unique) decomposition $\alpha_{\textnormal{s}}=\alpha_{\textnormal{d}}\alpha_{\textnormal{k}}$ where $\alpha_{\textnormal{d}}$ is diagonalizable purely real and $\alpha_{\textnormal{k}}$ is $\mathbf{C}$-diagonalizable with all eigenvalues on the unit circle. Finally write $\alpha_{\textnormal{d}}=\alpha_{\textnormal{p}}\alpha_{\pm}$ where $\alpha_{\textnormal{p}}$ is purely real and $\alpha_{\pm}$ has only eigenvalues in $\{-1,1\}$. Finally set $\alpha_{\textnormal{r}}=\alpha_{\textnormal{p}}\alpha_{\textnormal{u}}$ and $\alpha_{\textnormal{i}}=\alpha_{\pm}\alpha_{\textnormal{k}}$. 
This shows the existence result as well as the additional statement.

Now let $\alpha=\beta_{\textnormal{r}}\beta_{\textnormal{i}}$ be another decomposition as in the statement of the lemma. So $\beta_{\textnormal{r}}^{-1}\alpha_{\textnormal{r}}=\beta_{\textnormal{i}}\alpha_{\textnormal{i}}^{-1}$. Since by assumption $[\beta_{\textnormal{r}},\beta_{\textnormal{i}}]=1$, we see that $\alpha$ centralizes both $\beta_{\textnormal{r}}$ and $\beta_{\textnormal{i}}$. Because of the additional feature of $\alpha_{\textnormal{r}}$ and $\alpha_{\textnormal{i}}$, they also centralize $\beta_{\textnormal{r}}$ and $\beta_{\textnormal{i}}$. So $\beta_{\textnormal{r}}^{-1}\alpha_{\textnormal{r}}$ has all its eigenvalues positive real, while 
$\beta_{\textnormal{i}}\alpha_{\textnormal{i}}^{-1}$ has all its eigenvalues on the unit circle; thus both are unipotent; since the latter is $\mathbf{C}$-diagonalizable, it is the identity, and thus we deduce $\alpha_{\textnormal{i}}=\beta_{\textnormal{i}}$ and $\alpha_{\textnormal{r}}=\beta_{\textnormal{r}}$.
\end{proof}

If $G$ is a connected Lie group, let $\Nil(G)$ be its connected nilpotent radical, that is, its largest nilpotent normal connected subgroup; this is a closed characteristic subgroup. If $G$ is any LC-group such that $G^\circ$ is 
a Lie group (e.g., $G$ is a W-faithful focal LC-group), we define $\Nilc(G)=\Nil(G^\circ)$.

\begin{lem}\label{nira}
Let $G=N\rtimes K$ be a connected Lie group, where $N$ is simply connected nilpotent and $K$ is compact. Assume that $K$ acts faithfully on $N$. Then 
\begin{enumerate}
\item every cocompact nilpotent subgroup of $G$ is contained in $N$;
\item for every closed connected cocompact subgroup of $G$ with a decomposition $N_1\rtimes K_1$ with $N_1$ nilpotent and $K_1$ compact, we have $N_1=N$.
\end{enumerate}
\end{lem}
\begin{proof}
Let $\Gamma$ be a cocompact nilpotent subgroup of $G$. We can suppose that the projection of $\Gamma$ on $K$ is dense, so we have to show $K=1$. Assume the contrary. The action of $K$ on $N$ is nontrivial, so its action on $N/[N,N]$ is nontrivial as well; so we can mod out and thus assume $N$ is a vector group. Since $K$ is a compact connected Lie group with a dense nilpotent subgroup, it is abelian and in particular, modding out once more, we can suppose that $N$ is 2-dimensional with a nontrivial action of $K$. We can also mod out by the kernel of the $K$-action, and hence we are reduced to the case of $G=\mathbf{R}^2\rtimes \textnormal{SO}(2)$.
The Zariski connected subgroups therein are either contained in $\mathbf{R}^2$, or compact, or $G$ itself. Since the Zariski closure of $G$ is nilpotent and cocompact, the only possibility is that $\Gamma\subset\mathbf{R}^2$, contradicting that $\Gamma$ has a dense projection to $\textnormal{SO}(2)$.

Let us now prove the second statement. By the previous assertion, $N_1\subset N$. Hence $N_1$ is a closed connected cocompact subgroup of $N$, so $N_1=N$.
\end{proof}

\begin{lem}\label{nilc}
Let $G$ be a W-faithful focal LC-group and let $M$ be the kernel of the modular map $\Delta:G\to\mathbf{R}_+$. Define $N=\Nilc(G)$. Then $N$ is a simply connected nilpotent Lie group; it is a closed cocompact subgroup of the Lie group $M^\circ$; the action of any $\xi\in\Delta^{-1}(\mathopen]1,+\infty\mathclose[)$ on $N$ is contracting. 
Moreover, any closed cocompact subgroup $H$ of $G$ contains $N$ and indeed satisfies $\Nilc(H)=\Nilc(G)$.
\end{lem}
\begin{proof}
It is a general elementary fact that in any LC-group, every nilpotent normal subgroup is contained in the kernel of the modular function. So $N=\Nil(M^\circ)$. 

If $\xi\in\Delta^{-1}(\mathopen]1,+\infty\mathclose[)$, then using \cite[Prop.~6.15 and Theorem 7.3(iv)]{CCMT}, $\xi$ acts as a compaction on $M$ and thus on $M^\circ$. By \cite[Proposition 6.9]{CCMT}, we deduce that $N$ is simply connected, is cocompact in $M^\circ$ and that $\xi$ acts as a contraction on $N$. 

If $G$ is focal, write $G^{\circ\circ}$ for $G^\circ\cap\Ker(\Delta_G)=\Ker(\Delta_G)^\circ$.

If $H$ is closed and cocompact in $G$, it follows from Lemma \ref{facgo} that $f(H^{\circ\circ})$ is closed cocompact in $G^{\circ\circ}$. Choose $\xi$ as above inside $H$. Then there are, by \cite[Proposition 6.9]{CCMT}, $\xi$-invariant decomposition $G^{\circ\circ}=\Nil(G^{\circ\circ})\rtimes K$ and $H^{\circ\circ}=\Nil(H^{\circ\circ})\rtimes L$ with $K,L$ compact. The action of $K$ and $L$ respectively are faithful because their kernel is normal and $\mathsf{W}(G)=\mathsf{W}(H)=\{1\}$. By Lemma \ref{nira}, we obtain that $\Nil(H^{\circ\circ})=\Nil(G^{\circ\circ})$. Clearly, $\Nil(H^{\circ\circ})=\Nilc(H)$ and $\Nil(G^{\circ\circ})=\Nilc(G)$.
\end{proof}

Let $G$ be a W-faithful focal LC-group and $N=\Nilc(G)$.  Note that $\Out(N)$ is naturally the group of real points of a real linear algebraic group\footnote{It follows from basic Galois cohomology --~namely, the vanishing of the Galois cohomology in degree one of unipotent groups in characteristic zero~-- that the canonical homomorphism $\Out(N(K))\to(\Out(N))(K)$ is an isomorphism for every field $K$ containing a definition field of $N$, where $\Out(N)$ in the right-hand term is defined, as the algebraic group quotient of $\Aut(N)$ by the Zariski-closed subgroup of inner automorphisms, which is unipotent. Here we identify $\Out(N)$ with $(\Out(N))(\mathbf{R})$ and the previous remark shows this identification is harmless.}. Let $\mathsf{P}_G$ be the image of $G$ in $\Out(N)$. Let $\mathsf{P}_G^{\textnormal{r}}$ denote the set of $\alpha_{\textnormal{r}}$ when $\alpha$ ranges over $\mathsf{P}_G$.

If $\xi\in G$, let $\alpha(\xi)$ be the outer automorphism it induces on $N$.

\begin{lem}\label{l_bg}
Let $G$ be a W-faithful focal LC-group. The subgroup $\mathsf{P}_G\subset\Out(N)$ is closed and for every $\xi\notin\Ker(\Delta)$ it contains the subgroup $\langle\alpha(\xi)\rangle$ as an infinite cyclic discrete cocompact subgroup; besides, the subgroup $\mathsf{P}_G^{\textnormal{r}}\subset\Out(N)$ is closed and contains the infinite cyclic subgroup $\langle\alpha_{\textnormal{r}}\rangle$ as an infinite cyclic discrete cocompact subgroup for all $\alpha\notin\Ker(\Delta)$. .
\end{lem}
\begin{proof}
The natural homomorphism $G\to\Out(N)$ is continuous and factors through $G/N$, which contains the infinite cyclic subgroups $\langle\alpha\rangle$ and $\langle\alpha_{\textnormal{r}}\rangle$ as cocompact discrete subgroups; in restriction to those, it is proper, and therefore the homomorphism $G/N\to\Out(N)$ is proper and in particular has a closed image. 

The group $G/N$ is the direct product of a cyclic group and a compact group \cite[Proposition 6.9]{CCMT}. Therefore $\mathsf{P}_G$ is generated by a compact group $K$ and a cyclic group centralizing each other. It follows that the Zariski closure $S$ of $\mathsf{P}_G$ is generated by $K$ and an abelian group $A$ contained in the centralizer of $K$. According to Lemma \ref{dec_ri}, we can write $A=L\times P$ where $L$ is compact and $P$ is purely positive real. It follows that $S$ is the direct product of the compact group $LK$ and $P$, and $\rho:\alpha\mapsto\alpha_{\textnormal{r}}$ is the projection onto $P$. We see that elements of $P$ are exactly those elements in $S$ that are purely positive real. In particular $\mathsf{P}_G^{\textnormal{r}}\subset P$.  
Since $\rho$ is proper and $\langle\alpha\rangle$ is discrete infinite cyclic and cocompact in $\mathsf{P}_G$, we deduce that $\mathsf{P}_G^{\textnormal{r}}$ is closed and
contains the infinite cyclic subgroup $\langle\alpha_{\textnormal{r}}\rangle$ as an infinite cyclic discrete cocompact subgroup for all $\alpha\notin\Ker(\Delta)$. 
\end{proof}

Since $\mathsf{P}_G^{\textnormal{r}}$ is closed by Lemma \ref{l_bg}, we can introduce (see  Definition \ref{l_par} for the definition of $\Par$):

\begin{defn}
Define the closed subgroup $[\mathsf{P}_G]\subset\Out(N)$ to be equal to $\Par(\mathsf{P}_G^{\textnormal{r}})$.
\end{defn}

\begin{lem}\label{l_bgc}
The subgroup $[\mathsf{P}_G]\subset \Out(N)$ is isomorphic to $\mathbf{R}$ and centralizes $\mathsf{P}_G$.  Moreover for any closed cocompact subgroup $H\subset G$ we have $[\mathsf{P}_H]=[\mathsf{P}_G]$.
\end{lem}
\begin{proof}
By Lemma \ref{l_par}, we have $\Par(\mathsf{P}_G^{\textnormal{r}})\subset P$ and the inclusion $\mathsf{P}_G^{\textnormal{r}}\subset\Par(\mathsf{P}_G^{\textnormal{r}})$ is cocompact. By Lemma \ref{l_bg}, $\mathsf{P}_G^{\textnormal{r}}$ has an infinite cyclic discrete cocompact subgroup. Thus $\Par(\mathsf{P}_G^{\textnormal{r}})$ is a vector group with an infinite cyclic discrete cocompact subgroup, hence is isomorphic to $\mathbf{R}$. 

Finally observe that $[\mathsf{P}_G]=\Par(\langle\alpha_{\textnormal{r}}\rangle)$ for $\alpha=\alpha(\xi)$ and any $\xi\notin\Ker(\Delta)$. Picking $\xi\in H$, we obtain the same description for $[\mathsf{P}_H]$ and thus $[\mathsf{P}_G]=[\mathsf{P}_H]$.
\end{proof}

If $U_1,U_2$ are groups and $\phi:U_1\to U_2$ is an isomorphism, it induces an isomorphism $\Aut(U_1)\to\Aut(U_2)$ given by $\alpha\mapsto \phi\circ\alpha\circ\phi^{-1}$, which maps inner automorphism of $U_1$ onto inner automorphisms of $U_2$ and thus factors through an isomorphism $\phi_*:\Out(U_1)\to\Out(U_2)$.

\begin{lem}\label{cckfo}
Let $G_1,G_2$ be W-faithful focal LC-groups of connected type. Define $N_i=\Nilc(G_i)$. Equivalences:
\begin{enumerate}
\item\label{fock1} $G_1$ and $G_2$ are commable; 
\item\label{fock2} there exists a W-faithful focal LC-group $G$ and copci homomorphisms $G_1\to G\leftarrow G_2$;
\item\label{fock3} there exists an isomorphism $\phi:N_1\to N_2$ such that $\phi_*([\mathsf{P}_{G_1}])=[\mathsf{P}_{G_2}]$.  
\end{enumerate}
\end{lem}
\begin{proof}
(\ref{fock2}) trivially implies (\ref{fock1}).

Suppose (\ref{fock1}) and let us prove (\ref{fock3}). Since the condition in (\ref{fock3}) is transitive, we can suppose that there exists a copci homomorphism $\phi:G_1\to G_2$. By Lemma \ref{nilc}, $\phi$ restricts to an isomorphism $N_1\to N_2$ and thus also induces an isomorphism $\phi_*:\Out(N_1)\to\Out(N_2)$ mapping $\mathsf{P}_{G_1}$ to $\mathsf{P}_{G_2}$. Since $[\mathsf{P}_{G_i}]$ is canonically defined from $\mathsf{P}_{G_i}$ within $\Out(N_i)$, it follows that $\phi_*([\mathsf{P}_{G_1}])=[\mathsf{P}_{G_2}]$.

Suppose (\ref{fock3}) and let us show (\ref{fock2}). Identifying $N_1$ and $N_2$ through $\phi$, we can suppose that $N_1=N_2$ and $[\mathsf{P}_{G_1}]=[\mathsf{P}_{G_2}]$. Let $C$ be the centralizer of $B=[\mathsf{P}_{G_1}]=[\mathsf{P}_{G_2}]$. By Lemma \ref{l_bgc}, $\mathsf{P}_{G_1}$ and $\mathsf{P}_{G_2}$ are both contained in $C$. Since $C$ is Zariski-closed, it has finitely many ordinary connected components and $C/B$ as well. So if $L/B$ is a maximal compact normal subgroup of $C/B$, then for each $i=1,2$ there exists $g_i\in C$ such that $g_i\mathsf{P}_{G_i}g_i^{-1}\subset L$. Since $\mathsf{P}_{G_i}$ and $L$ are both 2-ended, the latter inclusion is cocompact. Therefore if $G$ is the inverse image of $L$ in $\Aut(N)$ if we identify $G_1$ and $G_2$ with their image in $\Aut(N)$ and if we lift $g_i$ to an element $\gamma_i\in\Aut(N)$, we have the cocompact inclusion $\gamma_iG_i\gamma_i^{-1}\subset L$ and $L$ is W-faithful focal.
\end{proof}

We say that a focal LC-group $G$ is focal-universal if it is W-faithful, and if for every focal LC-group $H$ commable to $G$, there is a copci homomorphism $H\to G$.

\begin{prop}\label{cu}
Let $G$ be a W-faithful focal LC-group of connected type and $N=\Nilc(G)$. The following statement are equivalent:
\begin{enumerate}
\item\label{un1} $G$ is focal-universal;
\item\label{un2} we have $[\mathsf{P}_G]\subset \mathsf{P}_G$ and $\mathsf{P}_G/[\mathsf{P}_G]$ is a maximal compact subgroup in $C([\mathsf{P}_G])/[\mathsf{P}_G]$, where $C(\cdot)$ denotes the centralizer in $\Out(N)$;
\item\label{un3} the image of $G$ in $\Out(N)$ is maximal among closed 2-ended subgroups.
\item\label{un15} $G$ is focal-universal within $\Aut(N)$, in the sense that, identifying $G$ with its image in $\Aut(N)$, it is not strictly contained as a cocompact subgroup in any larger closed subgroup of $\Aut(N)$;
\item\label{un4} every copci homomorphism $G\to G_1$ into a W-faithful focal LC-group is a topological isomorphism.
\end{enumerate}
\end{prop}
\begin{proof}
That (\ref{un3}) implies (\ref{un2}) is essentially clear: if $L/[\mathsf{P}_G]$ is a compact subgroup of $C([\mathsf{P}_G])/[\mathsf{P}_G]$ containing $\mathsf{P}_G[\mathsf{P}_G]/[\mathsf{P}_G]$, then $L$ is closed 2-ended and contains $\mathsf{P}_G$ so by (\ref{un3}) we have $\mathsf{P}_G=\mathsf{P}_G[\mathsf{P}_G]=L$, i.e.\ $[\mathsf{P}_G]\subset \mathsf{P}_G$ and $\mathsf{P}_G/[\mathsf{P}_G]$ is maximal compact in $C([\mathsf{P}_G])$.

Suppose (\ref{un15}) and let us show (\ref{un3}). Let the image $\mathsf{P}_G$ of $G$ in $\Out(N)$ be contained in a closed 2-ended subgroup $L$, and let $G_1$ be its inverse image in $\Aut(N)$. Then $G\subset G_1$ is cocompact so by (\ref{un15}), we have $G=G_1$ and hence $\mathsf{P}_G=L$.

Suppose (\ref{un4}) and let us show (\ref{un15}). Suppose that $G\subset G_1\subset\Aut(N)$. Then $G_1$ is W-faithful: indeed if $W$ is compact normal in $G_1$, then $W$ is mapped injectively into $\Out(N)$, and hence acts faithfully on $N$ which implies $[W,N]\neq 1$, hence $W\cap N\neq 1$, which is a contradiction unless $W=1$. So from (\ref{un4}) it follows that $G_1=G$.

Suppose (\ref{un1}) and let us show (\ref{un4}). Let $i:G\to G_1$ be a homomorphism as in (\ref{un4}). Since $G$ is W-faithful, $i$ is injective. 
Since $G$ is focal-universal, there also exists a copci homomorphism $j:G_1\to G$, and since $G_1$ is W-faithful, $j$ is also injective. The existence of $i$ and $j$ imply $\dim(G)=\dim(G_1)$ and then $|\pi_0(G)|=|\pi_0(G_1)|$, and finally that $i$ and $j$ are topological isomorphisms.

Finally suppose (\ref{un2}) and let us show (\ref{un1}). Let $G_1$ be commable to $G$. Then Lemma \ref{cckfo} shows that there exists a focal LC-group $G_3$ with copci homomorphisms $G_1\to G_3\leftarrow G$. 
But following the proof of 
(\ref{fock3})$\Rightarrow$(\ref{fock2}) of Lemma \ref{cckfo}, we see that the maximality of $\mathsf{P}_G$ implies that by construction the homomorphism $G\to G_3$ is an isomorphism. Thus we have a homomorphism $G_1\to G$ with the desired properties.
\end{proof}

\begin{cor}\label{efou}
Every focal LC-group of connected type $G$ admits a copci homomorphism into a W-faithful focal-universal LC-group $\hat{G}$. Moreover, $\hat{G}$ is uniquely defined up to isomorphism by the weaker requirement that it is focal-universal and that $G$ is commable to $\hat{G}$.
\end{cor}
\begin{proof}Again, we denote by $C(\cdot)$ the centralizer in $\Out(N)$.
We can suppose that $G$ is W-faithful; set $N=\Nilc(G)$. Being Zariski closed, $C([\mathsf{P}_G])$ has finitely many ordinary connected components and hence $C([\mathsf{P}_G])/[\mathsf{P}_G]$ as well. Let $L$ be a maximal compact subgroup of $C([\mathsf{P}_G])/[\mathsf{P}_G]$ containing $\mathsf{P}_G[\mathsf{P}_G]/[\mathsf{P}_G]$ and $\hat{G}$ the inverse image of $L$ in $\Aut(N)$. Then $[\mathsf{P}_G]=[\mathsf{P}_{\hat{G}}]$. By Proposition \ref{cu}, $\hat{G}$ is focal-universal.

The uniqueness follows from the definition of being focal-universal, and (\ref{un4}) of Proposition \ref{cu}.
\end{proof}

\begin{cor}\label{ckco}
Any two commable focal LC-groups of connected type are commable within focal groups through $\nearrow\nwarrow$.\qed
\end{cor}

If $G$ is a W-faithful focal LC-group of connected type and $N=\Nilc(G)$. Let $\check{G}$ be the inverse image of $[\mathsf{P}_G]$ in $\Aut(N)$. Call it the purely real core of $G$. Note that if $G$ is focal-universal, then $\check{G}$ is contained in $G$, and is normal and cocompact. 

\begin{defn}
A {\em Heintze group} is a Lie group of the form $N\rtimes\mathbf{R}$, where $N$ is a simply connected nilpotent Lie group and the action of positive elements of $\mathbf{R}$ on $N$ is contracting. It is {\em purely real} if all complex eigenvalues of the action of $\mathbf{R}$ on the Lie algebra of $N$ are real.
\end{defn}

\begin{remark}
Heintze \cite{Hein} proved that a connected Lie group of dimension $\ge 2$ admits a left-invariant Riemannian metric of negative sectional curvature if and only if it is a Heintze group.
\end{remark}

\begin{prop}\label{chec}
Let $G$ be a W-faithful focal LC-group of connected type. If $G$ is focal-universal, or more generally if $[\mathsf{P}_G]\subset \mathsf{P}_G$ then $\check{G}$ is the unique purely real Heintze closed cocompact subgroup in $G$. 
\end{prop}
\begin{proof}
Clearly $\check{G}$ satisfies these properties. Conversely, assume that $H$ satisfies these properties. By Lemma \ref{nilc}, $H$ contains $N=\Nilc(G)$. Since $H$ is a Heintze group, $H/N$ is isomorphic to $\mathbf{R}$. Since $H$ is purely real, we have $\mathsf{P}_H\subset[\mathsf{P}_G]$. Since $\mathsf{P}_H$ is isomorphic to $\mathbf{R}$, this is an equality and hence $H=\check{G}$.
\end{proof}

\begin{cor}\label{c_comh}
Every focal LC-group of connected type is commable to a purely real Heintze group, unique up to isomorphism.\qed
\end{cor}

\section{Commability in mixed type and the $\varpi$-invariant}

\begin{defn}\label{dvarpi}Let $G$ be a focal LC-group. Consider the modular functions of $G/(G^\circ\cap\Ker\Delta_G)$ (which is the same as $G/G^\circ$ if $G$ is not of connected type) and $G/G^\sharp$; by composition they define homomorphisms $\Delta_G^{\textnormal{td}}$, $\Delta_G^{\textnormal{con}}:G\to\mathbf{R}_+$, which we call restricted modular functions.
Since $\Hom(G,\mathbf{R})$ is 1-dimensional, if $G$ is not of totally disconnected type then $\Delta_G^{\textnormal{con}}$ is nontrivial and hence there exists a unique $\varpi=\varpi(G)\in\mathbf{R}$ such that $\log\circ\Delta_G^{\textnormal{td}}=\varpi(\log\circ\Delta_G^{\textnormal{con}})$. Because of the compacting element in $G$, necessarily $\varpi(G)\ge 0$, with equality if and only if $G$ is of connected type. If $G$ is of totally disconnected type we set $\varpi(G)=+\infty$.
\end{defn}

\begin{defn}\label{dhpa}Let $H,A$ be focal LC-groups, with $H$ of connected type with a surjective modular function and with $A$ of totally disconnected type. For $\varpi>0$, define
\[H\stackrel{\varpi}{\times}A=\{(x,y)\in H\times A\mid \Delta_H(x)^\varpi=\Delta_{A}(y)\}.\]
This is a focal LC-group of mixed type, satisfying $\varpi(H\stackrel{\varpi}{\times}A)=\varpi$. If $q\ge 2$ is an integer and $\varpi$ is a positive real number, define in particular
\[H[\varpi,q]=H\stackrel{\varpi}{\times}\FT_q.\]
\end{defn}

\begin{prop}\label{micom}
Let $G_1,G_2$ be focal LC-groups of mixed type. Equivalences:
\begin{enumerate}
\item\label{mi1} $G_1$ and $G_2$ are commable;
\item\label{gdgsv} the following three properties hold:
\begin{itemize}
	\item $G_1/G_1^\sharp$ and $G_2/G_2^\sharp$ are commable within focal groups
	\item $q_{G_1}=q_{G_2}$ (that is, $G_1/G_1^\circ$ and $G_2/G_2^\circ$ are commable within focal groups, see Corollary \ref{tdckfo2});
\item $\varpi(G_1)=\varpi(G_2)$;\end{itemize}
\item\label{mi3} there is a commation within focal groups $G_1\nearrow\nwarrow\nearrow\nwarrow G_2$;
\item\label{mi5} there exists a non-power integer $q\ge 2$, an integer $n\ge 1$, a focal-universal group of connected type $H$ and a positive real number $\varpi>0$ such that for $i=1,2$ there is a commation within focal groups $G_i\nearrow\nwarrow H[q,\varpi]^{[n]}$;
\item\label{mi6} there exists a non-power integer $q\ge 2$, a focal-universal group of connected type $H$ and a positive real number $\varpi>0$ such that for $i=1,2$, there is a commation within focal groups with $G_i\nearrow\nwarrow\nearrow H[\varpi,q]$;
\item\label{mi4} there is a commation within focal groups $G_1\nwarrow\nearrow\nwarrow\nearrow G_2$;
\item\label{mi7} there exists an integer $m\ge 2$, a focal-universal group of connected type $H$ and a positive real number $\varpi>0$ such that for $i=1,2$ there is a there is a commation within focal groups $G_i\nwarrow\nearrow H[\varpi,m]$.
\end{enumerate}
\end{prop}
\begin{proof}
Obviously any of (\ref{mi3}), (\ref{mi5}), (\ref{mi6}), (\ref{mi4}), or (\ref{mi7}) implies (\ref{mi1}).

(\ref{mi1})$\Rightarrow$(\ref{gdgsv}). When passing from $G$ to a closed cocompact subgroup or to the quotient modulo a compact normal subgroup, the restricted modular functions do not change. 
Therefore, assuming (\ref{mi1}), we obtain $\varpi(G_1)=\varpi(G_2)$. The remaining conditions of (\ref{gdgsv}) follow from Lemma \ref{facgo}.

Other immediate implications are that (\ref{mi5}) implies (\ref{mi3}) and (\ref{mi6}), and that (\ref{mi7}) implies (\ref{mi4}).

Let us now prove that (\ref{gdgsv}) implies (\ref{mi5}) and then (\ref{mi7}). 

Consider a proper cocompact embedding of $G/N_i^\sharp$ into a focal-universal group $H_i$. By (\ref{gdgsv}) and Corollary (\ref{efou}), $H_1$ and $H_2$ are isomorphic, so we write $H=H_i$. The diagonal map $G_i\to H_i\times G/G^\circ$
has its image contained and cocompact in $H\stackrel{\varpi}{\times}(G/G^\circ)$ (where again by (\ref{gdgsv}) $\varpi$ does not depend on $i$). Thus we get a copci homomorphism $G\to H\stackrel{\varpi}{\times}(G/G^\circ)$.

Consider the copci homomorphisms $G_i/G_i^\circ\to \FT_{q^{n_i}}\leftarrow \FT_{q}^{[n]}$ as in the proof of (\ref{fccck4})$\Rightarrow$(\ref{fccck8}) of Corollary \ref{tdckfo2}, where $q$ is a non-power integer. Then these extend to copci homomorphisms
\[G_i\to H\stackrel{\varpi}{\times}(G/G^\circ)\to H[\varpi,q^{n_i}]\leftarrow H[\varpi,q]^{[n]}.\]
This proves (\ref{mi5}).

Consider now the copci homomorphisms $G_i/G_i^\circ\leftarrow (G_i/G_i^\circ)^{[n-n_i]}\to \FT_{q^n}$ from the proof of (\ref{fccck4})$\Rightarrow$(\ref{fccck7}) of Corollary \ref{tdckfo2}. These extend to copci homomorphisms
\[G_i\leftarrow G_i^{[n-n_i]}\to H[\varpi,q^n],\]
proving (\ref{mi7}).
\end{proof}

In view of Corollary \ref{c_comh}, we deduce:

\begin{cor}\label{c_comix}
Every focal LC-group of mixed type $G$ is commable to an LC-group of the form $H[\varpi,q]$, for some purely real Heintze group $H$ of dimension $\ge 2$ uniquely defined up to isomorphism, a unique $\varpi\in\mathopen]0,\infty\mathclose[$, and a unique non-power integer $q\ge 2$.\qed
\end{cor}

The corollary can also be stated for an arbitrary focal group. If $H$ is a group with a quotient homomorphism onto $\mathbf{R}$ with compact kernel, we just define $H[\infty,q]=\FT_q$. Also, if $H$ if a focal group of connected type, we define $H[0,1]=H$.
  
\begin{cor}\label{c_comix2}
Every focal LC-group $G$ is commable within focal groups to an LC-group of the form $H[\varpi,q]$, for some purely real Heintze group $H$ uniquely defined up to isomorphism, a unique $\varpi\in [0,\infty]$, and a unique non-power integer $q\ge 1$. Here, the triple $(H,\varpi,q)$ satisfies the compatibility condition: $\dim(H)=1$ if and only if $\varpi=0$ and $q=1$ if and only if $\varpi=\infty$.\qed
\end{cor}

\section{Commability between focal and general type groups}\label{speci}

Unlike in the previous sections, we allow ourselves to use (soft) arguments related to hyperbolicity.

Recall that a hyperbolic LC-group is {\em non-elementary} if it boundary is uncountable, or equivalently contains at least 3 points. We define a {\em focal-hyperbolic} LC-group as a non-elementary hyperbolic LC-group whose action on the boundary has a fixed point. A {\em general type hyperbolic} LC-group is by definition a non-elementary hyperbolic LC-group that is not focal-hyperbolic.

In the terminology of the present paper, we have

\begin{thm}[\cite{CCMT}] An LC-group is focal-hyperbolic if and only if it is focal.\qed
\end{thm}

We now make free use of this theorem, writing focal for focal-hyperbolic.

If $H\to G$ is a copci homomorphism between LC-groups, $G$ is hyperbolic (resp.\ non-elementary hyperbolic) if and only if $H$ is hyperbolic (resp.\ non-elementary hyperbolic), and if $G$ is focal, so is $H$; conversely if $H$ is focal, then $G$ is either focal or of general type.

This situation was clarified in \cite{CCMT}. Define an S1-{\em group} to be an LC-group $G$ with $\mathsf{W}(G)$ compact, such that $G/\mathsf{W}(G)$ is isomorphic to an open subgroup of finite index in the isometry group of a rank 1 symmetric space (actually this index can only be 1 or 2). An S1-group is hyperbolic of general type.

\begin{prop}\label{specalt}
Let $H\to G$ be a copci homomorphism between locally compact groups, with $H$ focal and $G$ non-focal. Then either
\begin{itemize}
\item $H$ is of totally disconnected type, or
\item $H$ is of connected type, and $G$ is an S1-group.
\end{itemize} 
Conversely, if $H$ is a focal LC-group of totally disconnected type, then is admits a copci automorphism into a non-focal LC-group.
\end{prop}
\begin{proof}
Necessarily $G$ is of general type, so by the (elementary) \cite[Proposition 5.10]{CCMT}, either $G$ is compact-by-(totally disconnected) in which case $H$ is of totally disconnected type, or $G$ is an S1-group. In the latter case, $G/\mathsf{W}(G)$ is a Lie group and hence $H/\mathsf{W}(H)$ is a Lie group as well, so $H$ is of connected type.

For the second statement, we know that $H$ admits a copci homomorphism into $\FT_m$ for some $m\ge 2$, and the latter admits a copci homomorphism into the automorphism group of a $(m+1)$-regular tree, which is not focal (for instance, because it is admits cocompact lattices and hence is unimodular).
\end{proof}

\begin{defn}
We say that an LC-group is special focal if it is focal and admits a copci homomorphism into a non-focal LC-group.
\end{defn}

\begin{prop}
If $H\to G$ is a copci homomorphism between focal LC-groups, then $H$ is special focal if and only if $G$ is special focal. In other words, the class of special focal LC-groups is invariant under commability within focal LC-groups.
\end{prop}
\begin{proof}
Trivially $G$ special implies $H$ special. For the converse, assume $H$ special. By Proposition \ref{specalt}, it is either of totally disconnected or connected type. In the first case, $G$ is also special by the same proposition. The remaining case is when $H$ is special of connected type; in this case
we only sketch a proof, since the conclusion can also be deduced from quasi-isometric rigidity arguments.

The first step is to show that for every rank 1 symmetric space $X$ and boundary point $\omega$, the stabilizer $G$ of $\omega$ in $S=\Isom(X)$ is focal-universal, using the characterization in Proposition \ref{cu}(\ref{un2}). This can be done by a case-by-case check. In the language of semisimple groups, the setting is the following: $S$ is the automorphism group of a simple adjoint Lie group of real rank 1 (which can be identified to $S^\circ$); $G$ is a maximal parabolic subgroup; it admits a Levi decomposition $G=N\rtimes A$, where $N=\Nilc(G)$ is the unipotent radical of $G$ and $A$ has the canonical decomposition $A=[\mathsf{P}_G]\times K$ ($K$ compact). In each case, we have to check that, denoting by $C([\mathsf{P}_G])$ the centralizer of $[\mathsf{P}_G]$ in $\Aut(U)$, the subgroup $K$ is maximal compact in $C([\mathsf{P}_G])/[\mathsf{P}_G]$.

For instance, for the $n$-dimensional real hyperbolic space ($n\ge 2$), this group $G$ is isomorphic to the group of affine similarities of $\mathbf{R}^n$, namely $\mathbf{R}^n\rtimes (\textnormal{O}(n)\mathbf{R}^*)$. So $N=\mathbf{R}^n$, $\Out(N)=\GL_n(\mathbf{R})$, and $\mathsf{P}_G$ is the group of linear similarities of $\mathbf{R}^n$, $[\mathsf{P}_G]$ is the group of positive scalar matrices in $\GL_n(\mathbf{R})$. So $[\mathsf{P}_G]\subset \mathsf{P}_G$ and $\mathsf{P}_G/[\mathsf{P}_G]=\textnormal{O}(n)$ is a maximal compact subgroup in $C([\mathsf{P}_G])/[\mathsf{P}_G]=\textnormal{GL}_n(\mathbf{R})/\mathbf{R}_+$. The interested reader can check the other cases.

Hence assuming $H\to G$ is copci, and $H$ is special, let $H\to G_1$ be a copci homomorphism with $G_1$ non-focal. By Proposition \ref{specalt}, $G_1$ can be chosen to be the isometry group $\Isom(X)$ of a symmetric space $X$ of rank 1 and noncompact type. By amenability, $G_1$ fixes a boundary point, and hence falls into $\Isom(X)_\omega$ for some $\omega$. By the case-by-case verification, $\Isom(X)_\omega$ is focal universal. So by Theorem \ref{classcomcon}, there exists a copci homomorphism $G\to\Isom(X)_\omega$, and hence by composition we get a copci homomorphism $G\to\Isom(X)$. Thus $G$ is special. 
\end{proof}

The following lemma is essentially contained in \cite{CCMT}.

\begin{lem}\label{g12f}
Let $G_1,G_2$ be focal LC-groups. Equivalences:
\begin{enumerate}
\item\label{h12} there exists a hyperbolic LC-group $G$ with copci homomorphisms $G_1\to G\leftarrow G_2$. 
\item\label{hf12} there exists a focal LC-group $G$ with copci homomorphisms $G_1\to G\leftarrow G_2$. 
\end{enumerate}
\end{lem}
\begin{proof}
Let us prove the nontrivial implication. 
Assuming (\ref{h12}), if $G$ is focal, there is nothing to prove. Otherwise the  action of $G_i$ on the boundary of $G$ has a unique fixed point $\omega_i$ and is transitive on the complement \cite[Proposition 5.5(c)]{CCMT}; the action of $G$ on $\partial G$ extends the latter action and does not fix $\omega_i$, so $G$ is transitive on $\partial G$. Thus the stabilizers $G_{\omega_1}$ and $G_{\omega_2}$ are closed cocompact focal subgroups of $G$ and are conjugate. So after conjugation in $G$, we can suppose that $G_1$ and $G_2$ are mapped into a single closed cocompact focal subgroup of $G$, proving (\ref{hf12}).
\end{proof}

\begin{prop}\label{tdck}
Let $G_1,G_2$ be focal LC-groups of totally disconnected type. Then $G_1$ and $G_2$ are commable through $\nearrow\nwarrow\nearrow\nwarrow$, more precisely there exists $k\ge 2$ such that both $G_1$ and $G_2$ are commable to a finitely generated free group of rank $k$ through $\nearrow\nwarrow$.
\end{prop}
\begin{proof}
There is a copci homomorphism $G_i\to\Aut(T_i)$ for some regular tree $T_i$ of finite valency at least 3. For $d$ large enough, $\Aut(T_i)$ ($i=1,2$) contains a cocompact lattice isomorphic to a free group of rank $k=1+d!$. Thus we have copci homomorphisms $G_1\to\Aut(T_1)\leftarrow F_k\to \Aut(T_2)\leftarrow G_2$.
\end{proof}

Thus there exists a single commability class containing all focal LC-groups of totally disconnected type and all virtually free discrete groups.

\begin{que}\label{qtre}
Let $G$ be a locally compact group having a continuous, proper cocompact action on a locally finite tree $T$. Assume that $T$ has no vertex of valency 1 and has at least one vertex of degree $\ge 3$. Is $G$ commable to $F_2$ (or equivalently to a focal group of totally disconnected type)?
\end{que}

By a result of Bass and Kulkarni if $G$ is as in Question \ref{qtre} and is unimodular, then it has a cocompact lattice; in particular, the answer to Question \ref{qtre} is positive if $G$ is unimodular.

A thorough study of cocompact isometry groups of bounded valency trees is carried out in \cite{MSW02}. 

Note that the converse of Question \ref{qtre} holds: if a locally compact group $G$ is commable to $F_2$, then it admits a continuous, proper cocompact action on a locally finite tree $T$ with no vertex of valency 1 and has at least one vertex of degree $\ge 3$. This fact is used throughout \cite{MSW02} and is explicitly stated in the survey \cite{CQI}.

Using Lemma \ref{g12f}, we can improve Lemma \ref{copft}.

\begin{prop}\label{gimp}
For any two distinct integers $m,n\ge 2$, the groups $\FT_n$ and $\FT_m$ are not commable through $\nearrow\nwarrow$.
\end{prop}

When $m$ and $n$ have common powers, this is in a sense optimal, in view of Remark \ref{rsous}. Otherwise, we have:

\begin{prop}\label{gimp2}
For any two integers $m,n\ge 2$ such that $\sqrt[\max]{m}\neq\sqrt[\max]{n}$, there the groups $\FT_n$ and $\FT_m$ are not commable through $\nwarrow\nearrow\nwarrow$.
\end{prop}

\begin{proof}[Proof of Propositions \ref{gimp} and \ref{gimp2}]
In both cases, arguing by contradiction and up to switching $m$ and $n$, we can suppose we have copci homomorphisms between focal groups $\FT_n\leftarrow G \to H\leftarrow \FT_m$. Then $G$ is focal, so by Lemma \ref{g12f}, we can suppose that $H$ is focal. In view of Proposition \ref{copfo}, we can suppose that $H=\FT_s$ for some $s$, and by Lemma \ref{copft}, we deduce that $s=n$. Thus in both cases, we have copci homomorphisms $\FT_n\leftarrow G\to \FT_m$. 

In the setting of Proposition \ref{gimp}, we can assume from the beginning that $G=\FT_n$ and we obtain a contradiction with Lemma \ref{copft}.

In the setting of Proposition \ref{gimp2}, we deduce that $\FT_n$ and $\FT_m$ are commable within focal groups. Their q-invariants are $\sqrt[\max]{n}$ and $\sqrt[\max]{m}$ (see Lemma \ref{ftrm}), which should be equal by Proposition \ref{p_q}, a contradiction.
\end{proof}

As a corollary of Proposition \ref{gimp}, we obtain

\begin{cor}\label{nouni}
There are no focal-universal LC-groups of mixed or totally disconnected type.
\end{cor}
\begin{proof}
Let $G$ be a counterexample. First assume that $G$ is of totally disconnected type. Then for $q=q_G$, there are copci homomorphisms $\FT_q\to G\leftarrow \FT_{q^2}$. This contradicts Proposition \ref{gimp}.

The case of mixed type is based on the same idea and left to the reader.
\end{proof}


\section{QI-invariance of the $\varpi$-invariant}\label{qivpi}

In this section, we prove the following theorem.

\begin{thm}\label{th_mi}
Let $G_1,G_2$ be focal quasi-isometric LC-groups. Then $G_1$ and $G_2$ have the same type and
\begin{enumerate}
\item\label{com1} $G_1/G_1^\sharp$ and $G_2/G_2^\sharp$ are quasi-isometric;
\item\label{com2} 
$\varpi_{G_1}=\varpi_{G_2}$.
\end{enumerate}
\end{thm}

\begin{remark}\label{dyq}
T.~Dymarz proves in \cite{Dy12} that if $G_1$ and $G_2$ are focal LC-groups of mixed type then $q_{G_1}=q_{G_2}$. In combination with Theorem \ref{th_mi}, this provides a reduction of the quasi-isometric classification of focal groups to the connected type case, see the discussion in \cite{CQI}.
\end{remark}

We first need to describe the topology of the boundary.

\begin{prop}\label{bdmi}
Let $G$ be a focal LC-group and $\partial G$ its visual boundary; denote by $\omega$ the unique $G$-fixed point in $\partial G$. 
\begin{itemize}
\item if $G$ is of connected type $\partial G$ is homeomorphic to a $(d-1)$-sphere for some unique integer $d=d_G\ge 2$;
\item if $G$ is of totally disconnected type, $\partial G$ is homeomorphic to a Cantor space $\mathsf{C}$;
\item if $G$ is of mixed type, then $\partial G$ is homeomorphic to the 1-point compactification $\Xi_d$ of $\mathbf{R}^{d-1}\times \mathsf{C}\times \mathbf{Z}$, where $\omega$ corresponds to the compactifying point, for some unique integer $d=d_G\ge 2$. 
\end{itemize}
In particular, if $G$ is of mixed type, then every self-homeomorphism of $\partial G$ fixes $\omega$. If $H$ is a Heintze group and $G=H[\varpi,q]$ then $d_G=d_H=\dim(H)$.
\end{prop}

Note that there are some trivial instances of the topological description of these boundaries: the $(d-1)$-sphere is homeomorphic to the one-point compactification of $\mathbf{R}^{d-1}$, and the Cantor space $\mathsf{C}$ is homeomorphic to the one-point compactification of $\mathsf{C}\times\mathbf{Z}$.

\begin{proof}[Proof of Proposition \ref{bdmi}]
The uniqueness of $d$ follows from the fact that the dimension of a manifold is a topological invariant in the case of the $d$-sphere. In $\Xi_d$, observe that $\omega$ is the unique point whose complement is disconnected, and that any connected component of $\Xi_d\smallsetminus\{\omega\}$ is homeomorphic to $\mathbf{R}^d$. This proves the uniqueness of $d$ as well as the fact that $\omega$ is fixed by any self-homeomorphism of $\Xi_d$.

To prove the main statement, if $G$ has connected type, then by Proposition \ref{chec}, we can suppose that $G$ is Heintze. Hence it  admits a left-invariant Riemannian metric of negative sectional curvature and hence its boundary is homeomorphic to a $(\dim(H)-1)$-dimensional sphere. If $G$ has totally disconnected type, then it admits a continuous proper cocompact action on a regular tree of finite valency $\ge 3$ and hence its boundary is homeomorphic to a Cantor space.

Finally assume that $G$ has mixed type. We can suppose by Corollary \ref{c_comix} that $G=H[\varpi,m]$, with real parameter $\varpi\in\mathopen]0,\infty\mathclose[$ and integer $m\ge 2$. The latter admits, as a cocompact subgroup, a group of the form $(N\times\mathbf{Q}_m)\rtimes\mathbf{Z}$, where the positive generator of $\mathbf{Z}$ acts on the simply connected nilpotent Lie group $N$ by a contracting automorphism, and on $\mathbf{Q}_m$ by multiplication by $m$ (which is also contracting). The action of $N\times \mathbf{Q}_m$ on $\partial G\smallsetminus \{\omega\}$ is transitive and the stabilizer of a given point is a compact subgroup, invariant under the action of $\mathbf{Z}$ (see the proof of \cite[Proposition 5.5]{CCMT}); since the action of $\mathbf{Z}$ is contracting, this stabilizer is trivial. Hence the action is simply transitive and $\partial G\smallsetminus \{\omega\}$ is homeomorphic to $N\times\mathbf{Q}_m$. Since a compact Hausdorff space is homeomorphic to the one-point compactification of the complement of any singleton, we obtain the desired conclusion.
\end{proof}

We now proceed to the proof of Theorem \ref{th_mi}. We first prove (\ref{com1}), which is a simple consequence of the structure of the boundary. We next prove (\ref{com2}), using $L^p$-cohomology computations from \cite{CoTe}. 

\begin{proof}[Proof of Theorem \ref{th_mi}, first part]

Observe that by Proposition \ref{bdmi}, a focal LC-group has connected type if and only if its boundary is locally connected, and has totally disconnected type if and only if its boundary is totally disconnected. This proves that $G_1$ and $G_2$ have the same type, and reduces (\ref{com2}) to the mixed type case. 

Writing $H_i=G_i/G_i^\sharp$, let us prove (\ref{com1}). It is trivial in the connected type case; in the totally disconnected type case, both $H_1$ and $H_2$ are quasi-isometric to $\mathbf{Z}$. So we can suppose that $G_1$ and $G_2$ both have mixed type.

The quasi-isometry from $G_1$ to $G_2$ induces a quasi-symmetric homeomorphism $h:\bd G_1\to\bd G_2$. Let $\omega_i$ be the $G_i$-fixed point in $\bd G_i$. By Proposition \ref{bdmi}, there exists $d\ge 2$ such that $\bd G_i$ is homeomorphic to $\Xi_d$ for $i=1,2$, and by the same proposition, $h(\omega_1)=\omega_2$. Hence $h$ restricts to a quasi-symmetric homeomorphism from $\bd G_1\smallsetminus\{\omega_1\}$ to $\bd G_2\smallsetminus\{\omega_2\}$, which sends a connected component $C_1$ (fixed once and for all) of $\bd G_1\smallsetminus\{\omega_1\}$ to a connected component $C_2$ of $\bd G_2\smallsetminus\{\omega_2\}$. Observe that $C_1\cup\{\omega\}$ is quasi-symmetrically homeomorphic to $\partial H_i$ for $i=1,2$. 
Accordingly, $\bd H_1$ and $\bd H_2$ are quasi-symmetrically homeomorphic. Using \cite[Corollary 7.2.3]{BS07}, we deduce that $H_1$ and $H_2$ are quasi-isometric.

It remains to prove Theorem \ref{th_mi}(\ref{com2}) in the mixed type case; we pause for a little while.
\end{proof}

We use computations of $L^p$-cohomology carried out in \cite{CoTe} (inspired by Pansu's computations for Lie groups). Recall that for any locally compact group $G$ endowed with a left Haar measure, its first $L^p$-cohomology group $H_1^p(G)$ is defined as the 1-cohomology of its right regular representation. It was shown by Pansu in a different context, and in \cite[Appendix~B]{CoTe} in this setting, that for a given real number $p\ge 1$, the vanishing of $H^1_p(G)$ is a quasi-isometry invariant of $G$. In particular, the number $\inf\{p\ge 1:H^1_p(G)\neq 0\}\in [1,\infty]$ is a numerical quasi-isometry invariant of $G$.

The following theorem was obtained by Pansu in \cite{Pan07} in the Riemannian case, using a differentiable definition of $L^p$-cohomology.

\begin{thm}[Theorem 7 in \cite{CoTe}]\label{lpcote}
Consider a focal LC-group of the form $G=(M\times V)\rtimes\mathbf{Z}$, where the negative generator of $\mathbf{Z}$ acts as a compaction on $M\times V$, $M$ is a nontrivial simply connected nilpotent Lie group, and $V=G^\sharp$ is totally disconnected. Let the positive generator of $\mathbf{Z}$ multiply the volume of $M\times V$ by $\delta$ and let $\lambda>1$ be the smallest modulus of an eigenvalue for its action on $M$. Define $p_0(G)=\log(\delta)/\log(\lambda)$. Then for all real $p\ge 1$, we have $H^1_p(G)\neq 0$ if and only $p>p_0(G)$.\qed
\end{thm}

\begin{lem}\label{p0c}
Under the assumptions of Theorem \ref{lpcote}, we have \[p_0(G)=(1+\varpi(G))p_0(G/G^\sharp).\]
\end{lem}
\begin{proof}
Let the positive generator of $\mathbf{Z}$ multiply the volume of $M$ by $\delta'$ and of $V$ by $\delta''$, so that $\delta=\delta'\delta''$. Then $p_0(G/V)=\log(\delta')/\log(\lambda)$, and thus
\[p_0(G)/p_0(G/V)=\log(\delta)/\log(\delta')=1+\log(\delta'')/\log(\delta').\]
By definition of $\varpi$, we have $\varpi(G)=\log(\delta'')/\log(\delta')$, proving the lemma.
\end{proof}

\begin{proof}[Proof of Theorem \ref{th_mi}(\ref{com2}), conclusion]
First assume that $G_1$ and $G_2$ have the form given in Theorem \ref{lpcote}.
By Pansu's quasi-isometry invariance of $L^p$-cohomology and Theorem \ref{lpcote}, we have $p_0(G_1)=p_0(G_2)$. Besides, by Theorem \ref{th_mi}(\ref{com1}), 
$H$ is quasi-isometric to $H'$, and thus again by Pansu's quasi-isometry invariance of $L^p$-cohomology and Theorem \ref{lpcote} (this time in the connected type case due to Pansu), we have $p_0(H_1)=p_0(H_2)$ (recall that $H_i=G_i/G_i^\sharp$). By Lemma \ref{p0c}, we have
\[\varpi(G_1)=p_0(G_1)/p_0(H_1)=p_0(G_2)/p_0(H_2)=\varpi(G_2).\]
In general, $G_i$ is commable to a group $G'_i$ of the form given in Theorem \ref{lpcote}. So by commability invariance of $\varpi$ and the previous case, we have $\varpi(G_1)=\varpi(G'_1)=\varpi(G'_2)=\varpi(G_2)$. 
\end{proof}

Using Dymarz' result (see Remark \ref{dyq}), this yields the following corollary.

\begin{cor}\label{vqi}
Let $H$ be a Heintze group, and for $i=1,2$, let $\varpi_i$ be positive real numbers and $k_i$ be integers $\ge 2$. Then $G_1=H[\varpi_1,k_1]$ and $G_2=H[\varpi_2,k_2]$ are quasi-isometric if and only if $\varpi_1=\varpi_2$ and $\sqrt[\max]{k_1}=\sqrt[\max]{k_2}$.
\end{cor}
\begin{proof}
Since $s_{G_i}=k_i$, we have $q_{G_i}=\sqrt[\max]{k_i}$, Theorem \ref{th_mi} together with Dymarz' theorem (Remark \ref{dyq}) proves the hardest part, namely the ``only if" implication.

The ``if" part is easier. Indeed, if the conditions are satisfied, then $q_{G_1}=q_{G_1}$, $\varpi(G_1)=\varpi(G_2)$ and $G_1/G_1^\sharp$ and $G_2/G_2^\sharp$ are isomorphic, so by Proposition \ref{micom}, $G_1$ and $G_2$ are commable, hence quasi-isometric.
\end{proof}

Let us indicate the geometric interpretation of Corollary \ref{vqi}. Let $X$ be a homogeneous negatively curved Riemannian manifold of dimension $\ge 1$. If $k\ge 1$, the {\em millefeuille space} $X[k]$, introduced in \cite{CCMT}, is the space obtained as follows. Let $\omega$ be a boundary point of $X$ such that $\Isom(X)_\omega$ is transitive on $X$ (if $\Isom(X)$ is focal, $\omega$ is taken to be the fixed point; otherwise any boundary point works); let $b$ be a Busemann function on $X$ based at $\omega$. Let the $(k+1)$-regular tree $T_k$ be endowed with a Busemann function $b'$. Then the millefeuille space is defined as
\[X[k]=\{(x,y)\in X\times T_k:\;b(x)=b'(y)\}.\]
It is endowed with a natural geodesic distance, so that for any vertical geodesic $D$ in $T_k$, the subset $X[k]\cap (X\times D)$ is geodesic and the natural projection from $X[k]\cap (X\times D)$ to $X$ is a bijective isometry. If the sectional curvature of $X$ is $\le-\kappa$ then $X[k]$ is CAT($-\kappa$).

Also, for $t>0$, let $X_{\{t\}}$ be the Riemannian manifold obtained from $X$ by multiplying the Riemannian metric by $t^{-2}$ (thus multiplying the Riemannian distance by $t^{-1}$ and the curvature by $t$).

\begin{cor}\label{corvv}
Let $X$ be a homogeneous negatively curved Riemannian manifold of dimension $\ge 2$, and let $k_1,k_2$ be integers $\ge 2$ and let $t_1,t_2$ be positive real numbers. Then $X_{\{t_1\}}[k_1]$ and $X_{\{t_2\}}[k_2]$ are quasi-isometric if and only if $\log(k_1)/t_1=\log(k_1)/t_2$ and $\sqrt[\max]{k_1}=\sqrt[\max]{k_2}$. In particular:
\begin{enumerate}[(a)]
\item\label{vk} 
The millefeuille spaces $X[k]$, for $k\ge 1$, are pairwise not quasi-isometric.
\item\label{vr} When $k\ge 2$ is fixed, the $X_{\{t\}}[k]$, for $t>0$, are pairwise non-quasi-isometric.
\end{enumerate}
\end{cor}
\begin{proof}
If $Y$ is a proper metric space that admits a continuous proper cocompact action of a focal LC-group $G$, we can define $\varpi(Y)=\varpi(G)$; this does not depend on the choice of $G$. The latter, for $Y=X[k]$, can be chosen to be of the form
$(M\times V)\rtimes\mathbf{Z}$ ($M$ connected, $V$ totally disconnected, the negative generator of $\mathbf{Z}$ acting by compaction) and denoting by $\delta_0$ the multiplication of the volume of the positive generator of $\mathbf{Z}$ on $M$
We see that we have $\delta_0^\varpi=k$, so
\[\varpi(X[k])=\log(k)/\log(\delta_0).\]
If the Riemannian metric is multiplied by $t^{-2}$, the Riemannian distance is multiplied by $t^{-1}$ and the Busemann function (normalized to vanish at a given base-point) is also multiplied by $t^{-1}$; we need to modify the action of $\mathbf{R}$ on $N$ to obtain that the element $1\in\mathbf{R}$ shifts the Busemann function by 1, namely by precomposing the action by the homothety $u\mapsto tu$. This replaces $(\delta_0,\lambda)$ by $(\delta_0^t,\lambda^{t})$ and thus 
we obtain $(\delta_0^{t})^\varpi=k$, so 
\[\varpi(X_{\{t\}}[k])=\frac{\log(k)}{t\log(\delta_0)}\;,\]

Since $\varpi$ is a quasi-isometry invariant, we deduce that $\frac{\log(k)}{\sqrt{t}}$ is a quasi-isometry invariant among the $X_{\{t\}}[k]$; in particular, when either $t$ or $k$ is fixed, this varies injectively in the other variable. This shows the two particular cases (only relying on $L^p$-cohomology, not on Dymarz' theorem).

Now let $G_i$ be a focal LC-group acting properly cocompactly on $X_{\{t_i\}}[k_i]$. Then $q_{G_i}=\sqrt[\max]{k_i}$. So the equivalence follows from Corollary \ref{vqi} (thus relying on both $L^p$-cohomology and Dymarz' theorem).
\end{proof}

\begin{remark}
The computation also shows that $p_0(X_{\{t\}}[k]))=p_0(X)+\frac{\log(k)}{t\log(\lambda)}$.
\end{remark}

Let us also give a consequence of Theorem \ref{th_mi} and Dymarz' theorem in terms of Conjecture \ref{fconj}.

\begin{cor}
Let $G$ be a focal LC-group of mixed type. If the focal LC-group of
connected type $G/G^\sharp$ satisfies Conjecture \ref{fconj}, then so does $G$.
\end{cor}
\begin{proof}
Suppose that $G_1$ and $G_2$ are focal of mixed type and quasi-isometric. Then $G_1/G_1^\sharp$ and $G_2/G_2^\sharp$ are quasi-isometric by Theorem \ref{th_mi}. So they are commable. It follows that there exist a Heintze group $H$, positive real numbers $\varpi_i$ and non-square integers $k_i\ge 2$ such that $G_i$ is commable to $H[\varpi_i,k_i]$ for $i=1,2$. By Theorem \ref{th_mi} and Dymarz' theorem, we have $\varpi_1=\varpi_2$ and $k_1=k_2$. Hence $G_1$ and $G_2$ are commable.
\end{proof}

Let us also provide the quasi-isometry classification of millefeuille spaces for which the nilpotent radical of the connected part is abelian. For convenience, we state it in the language of groups. 

If $d\ge 2$, let $A$ be a $(d-1)\times (d-1)$ invertible real matrix all of whose eigenvalues are of modulus $<1$. Write $\FT_k=U_k\rtimes_\alpha\mathbf{Z}$, where $\alpha$ acts by compaction on $U_k$. Define the group
\[G(A,k)=(\mathbf{R}^{d-1}\times U_k)\rtimes_{(A,\alpha)}\mathbf{Z}.\]
Note that $\FT_1=\mathbf{Z}$, so $G(A,1)=\mathbf{R}^{d-1}\rtimes_A\mathbf{Z}$.

Combining our result, Xie's quasi-isometric classification of Heintze groups with abelian derived subgroup, and Dymarz' quasi-isometry invariance of $q$, we obtain the following corollary, also obtained by Dymarz \cite[Corollary 5]{Dy12} with a related but different approach.

\begin{cor}
For $i=1,2$, let $A_i$ be a square invertible matrix of size $d_i\ge 2$, with only real eigenvalues, all in $\mathopen] 0,1\mathclose[$, and let $k_i\ge 2$ be integers. Then the following are equivalent.
\begin{enumerate}[(i)]
\item\label{gaqi} $G(A_1,k_1)$ and $G(A_2,k_2)$ are quasi-isometric;
\item\label{gaco} $G(A_1,k_1)$ and $G(A_2,k_2)$ are commable;
\item\label{numb} there exist integers $n_1,n_2\ge 1$ such that $k_1^{n_1}=k_2^{n_2}$, and $A_1^{n_1}$ and $A_2^{n_2}$ are conjugate.
\end{enumerate}
\end{cor}
\begin{proof}
Write $G_i=G(A_i,k_i)$. We have $\varpi(G_i)=\log(k_i)/\log(\det(A_i))$ and $q(G_i)=\sqrt[\max]{k_i}$, and $G_i/G_i^\sharp=G(A_i,1)$.

(\ref{numb})$\Rightarrow$(\ref{gaco}) is straightforward (by a direct argument, or using Proposition \ref{micom}).

(\ref{gaco})$\Rightarrow$(\ref{gaqi}) is trivial.

Assume (\ref{gaqi}). Then by Dymarz' theorem, $q(G_1)=q(G_2)$, we denote it by $q$ and write $k_1=q^{n_2}$, $k_2=q^{n_1}$. By Theorem \ref{th_mi}, $\varpi(G_1)=\varpi(G_2)$, and $G_1/G_1^\sharp$ and $G_2/G_2^\sharp$ are quasi-isometric. By Xie's theorem \cite{xx}, there exists a unique $t>0$ such that $A_2$ is conjugate to $A_1^t$ (recall that $A_1$ is contained in a unique one-parameter subgroup consisting of matrices with positive eigenvalues). We can thus rewrite the equality $\varpi(G_1)=\varpi(G_2)$ as
\[\frac{n_2\log(q)}{\log(\det(A_1)}=\frac{n_1\log(q)}{t\log(\det(A_1)}.\]
Thus $t=n_1/n_2$. So $k_1^{n_1}=k_2^{n_2}$ and $A_1^{n_1}$ and $A_2^{n_2}$ are conjugate. Thus (\ref{numb}) holds.
\end{proof}


%

\end{document}